\documentclass{mrlart7}

\usepackage{amsmath,xcolor}
\usepackage{geometry,amsthm,graphics,tabularx,amssymb,shapepar}
\usepackage{amscd}
\usepackage{bm,bbm}
\usepackage{mathrsfs}
\usepackage{enumerate}
\usepackage{BOONDOX-cal}
\usepackage[all]{xypic}

\newtheorem{thm}{Theorem}[section]
\newtheorem{lem}[thm]{Lemma}
\newtheorem{defi}[thm]{Definition}
\newtheorem{prop}[thm]{Proposition}
\newtheorem{rem}[thm]{Remark}

\newtheorem{con}[thm]{Conjecture}

\allowdisplaybreaks

\numberwithin{equation}{section}
\title[Affine vertex operator superalgebra $L_{\widehat{sl(2|1)}}(\mathcal{k},0)$]
{Affine vertex operator superalgebra $L_{\widehat{sl(2|1)}}(\mathcal{k},0)$ at boundary admissible level}

\author{Huaimin Li}
\address{School of Mathematical Sciences, Xiamen University,
 Xiamen, China 361005} \email{huaiminli@stu.xmu.edu.cn}

\author{Qing Wang}
\address{School of Mathematical Sciences, Xiamen University,
 Xiamen, China 361005} \email{qingwang@xmu.edu.cn}

\begin{document}

\begin{abstract}
Let $L_{\widehat{sl(2|1)}}(\mathcal{k},0)$ be the simple affine vertex operator superalgebra
associated to the affine Lie superalgebra $\widehat{sl(2|1)}$ with level $\mathcal{k}$.
We conjecture that
$L_{\widehat{sl(2|1)}}(\mathcal{k},0)$ is rational in the category $\mathcal{O}$ at boundary admissible level $\mathcal{k}$ and
there are finitely many irreducible weak $L_{\widehat{sl(2|1)}}(\mathcal{k},0)$-modules in the category $\mathcal{O}$,
where the irreducible modules are exactly the admissible modules of level $\mathcal{k}$ for $\widehat{sl(2|1)}$.
In this paper, we first prove this conjecture at boundary admissible level $-\frac{1}{2}$. Then we give an example to show that outside of the boudary levels, $L_{\widehat{sl(2|1)}}(\mathcal{k},0)$ is not rational in the category $\mathcal{O}$.
Furthermore, we consider the $\mathbb{Q}$-graded vertex operator superalgebras $(L_{\widehat{sl(2|1)}}(\mathcal{k},0),\omega_\xi)$
associated to a family of new Virasoro elements $\omega_\xi$, where $0<\xi<1$ is a rational number.
We determine the Zhu's algebra $A_{\omega_\xi}(L_{\widehat{sl(2|1)}}(-\frac{1}{2},0))$ of $(L_{\widehat{sl(2|1)}}(-\frac{1}{2},0),\omega_\xi)$
and prove that $(L_{\widehat{sl(2|1)}}(-\frac{1}{2},0),\omega_\xi)$ is rational and $C_2$-cofinite.
Finally, we consider the case of non-boundary admissible level $\frac{1}{2}$ to support our conjecture, that is,
we show that there are infinitely many irreducible weak $L_{\widehat{sl(2|1)}}(\frac{1}{2},0)$-modules in the category $\mathcal{O}$
and $(L_{\widehat{sl(2|1)}}(\frac{1}{2},0),\omega_\xi)$ is not rational.
\end{abstract}
\maketitle

\section{Introduction}

In \cite{KW}, Kac and Wakimoto introduced the notion of admissible weights to study the modular invariant representations for affine Lie (super)algebras,
and furthermore they showed that the modular invariance property of normalized (super)characters of admissible modules holds for
affine Lie (super)algebras associated to any simple finite-dimensional Lie algebra or $osp(1|2n)$.
While for other basic Lie superalgebras,
modular invariance occurs in boundary level admissible modules \cite{GK15,KW17}.
In this paper, we study the admissible modules in context of vertex operator superalgebra.
Specially, we consider the Lie superalgebra $sl(2|1)$,
which together with $osp(1|2)$ can be seen as the building block of basic Lie superalgebras.
We conjecture that
$L_{\widehat{sl(2|1)}}(\mathcal{k},0)$ is rational in the category $\mathcal{O}$ at boundary admissible level $\mathcal{k}$
and the irreducible weak modules in the category $\mathcal{O}$ are exactly the admissible modules of level $\mathcal{k}$ for $\widehat{sl(2|1)}$.
In \cite{LW}, we considered a category $\mathcal{C}$ of $\widehat{osp(1|2)}$-modules which contains the category $\mathcal{O}$,
that is, the $\widehat{osp(1|2)}$-modules in the category $\mathcal{C}$ are those the positive part of $\widehat{osp(1|2)}$ acts locally nilpotently,
and we proved that $L_{\widehat{osp(1|2)}}(\mathcal{k},0)$ is rational in the category $\mathcal{C}$ at admissible level $\mathcal{k}$,
then the rationality of $L_{\widehat{osp(1|2)}}(\mathcal{k},0)$ at admissible level $\mathcal{k}$ in the category $\mathcal{O}$  follows immediately.
In \cite{W},
Wood also proved that $L_{\widehat{osp(1|2)}}(\mathcal{k},0)$ is rational in the category $\mathcal{O}$ at admissible level $\mathcal{k}$
and classified the irreducible relaxed highest weight modules for $L_{\widehat{osp(1|2)}}(\mathcal{k},0)$.
Then in \cite{GS},
Gorelik and Serganova
showed the rationality of $L_{\widehat{osp(1|2n)}}(\mathcal{k},0)$ at admissible level $\mathcal{k}$ in the category $\mathcal{O}$.
For the simple Lie algebra,
Adamovi\'{c} and Milas proved that the simple vertex operator algebra
$L_{\widehat{sl_2}}(\mathcal{k},0)$ is rational in the category $\mathcal{O}$ at admissible level $\mathcal{k}$ in \cite{AM},
Arakawa proved the rationality of $L_{\widehat{\mathfrak{g}}}(\mathcal{k},0)$ at admissible level $\mathcal{k}$ in the category $\mathcal{O}$
for any simple finite-dimensional Lie algebra $\mathfrak{g}$ in \cite{A}.
In this paper,
we prove our conjecture for $L_{\widehat{sl(2|1)}}(\mathcal{k},0)$ at the boundary admissible level $-\frac{1}{2}$.
In fact, the superalgebra $L_{\widehat{sl(2|1)}}(-\frac{1}{2},0)$ is an infinite-order simple current extension of $V_{\widehat{gl(1|1)}}(1,0)$ \cite{CR},
and a tensor category of $L_{\widehat{sl(2|1)}}(-\frac{1}{2},0)$-modules was obtained by applying the tensor supercategory structure
on the Kazhdan-Lusztig category of affine $gl(1|1)$ in \cite{CMY}.
The semisimplicity of the Kazhdan-Lusztig category has been proved for $L_{\widehat{sl(2|1)}}(\mathcal{k},0)$
at boundary admissible level $\mathcal{k}$ in \cite{AMP}.
Recently, Gorelik and Kac proved that for non-twisted affine Kac-Moody superalgebra $\mathfrak{g}$ and boundary admissible level $\mathcal{k}$,
the number of isomorphism classes of irreducible $L_{\mathfrak{g}}(\mathcal{k},0)$-modules in the category $\mathcal{O}$ is finite and
any $L_{\mathfrak{g}}(\mathcal{k},0)$-module in the category $\mathcal{O}$ is completely reducible in \cite{GK25}.

We know that simple affine vertex operator (super)algebras $L_{\widehat{\mathfrak{g}}}(\mathcal{k},0)$
at admissible level $\mathcal{k}$ may not be rational if $\mathcal{k}$ is not a positive integer.
In \cite{DLM2}, Dong, Li and Mason showed that $L_{\widehat{sl_2}}(\mathcal{k},0)$ at admissible level $\mathcal{k}$ is a
rational $\mathbb{Q}$-graded vertex operator algebra
under a new Virasoro element and irreducible ordinary modules are exactly the admissible modules of level $\mathcal{k}$ for $\widehat{sl_2}$.
Then these results have been proved by Lin for
$L_{\widehat{\mathfrak{g}}}(\mathcal{k},0)$ associated to any simple finite-dimensional Lie algebra $\mathfrak{g}$ at admissible level $\mathcal{k}$ in \cite{Lin}.
We showed that these results also hold for $L_{\widehat{osp(1|2)}}(\mathcal{k},0)$ at admissible level $\mathcal{k}$ in \cite{LW}.
In this paper, we show that
$L_{\widehat{sl(2|1)}}(\mathcal{k},0)$ at boundary admissible level $-\frac{1}{2}$
is a rational $\mathbb{Q}$-graded vertex operator superalgebra under a new Virasoro element.
However, under such kind of new Virasoro elements, $L_{\widehat{sl(2|1)}}(\mathcal{k},0)$ at non-boundary admissible level $\frac{1}{2}$
is not rational.

Let $\mathfrak{g}={sl}(2|1)$
and $\mathcal{O}_{\mathcal{k}}$ be the full subcategory of
weak $L_{\widehat{\mathfrak{g}}}(\mathcal{k},0)$-module category such that $M$ is an object of $\mathcal{O}_{\mathcal{k}}$
if and only if $M\in\mathcal{O}$ as $\widehat{\mathfrak{g}}$-module,
we obtain all irreducible weak $L_{\widehat{\mathfrak{g}}}(-\frac{1}{2},0)$-modules in the category $\mathcal{O}_{-\frac{1}{2}}$
and all irreducible ordinary $L_{\widehat{\mathfrak{g}}}(-\frac{1}{2},0)$-modules.
We prove that there are only finitely many irreducible modules in the category $\mathcal{O}_{-\frac{1}{2}}$ up to isomorphism
and the irreducible modules are exactly the admissible modules of level $-\frac{1}{2}$ for $\widehat{\mathfrak{g}}$.
Furthermore, we prove that the category $\mathcal{O}_{-\frac{1}{2}}$ is semisimple,
so any ordinary
$L_{\widehat{\mathfrak{g}}}(-\frac{1}{2},0)$-module is completely reducible.
More general, we consider the category $\mathcal{C}_{\mathcal{k}}$,
which is the full subcategory of the weak $L_{\widehat{\mathfrak{g}}}(\mathcal{k},0)$-module category
such that $M$ is an object in $\mathcal{C}_{\mathcal{k}}$ if and only if the sum of all positive root spaces of
$\widehat{\mathfrak{g}}$ acts locally nilpotently on $M$.
We prove that the category $\mathcal{C}_{-\frac{1}{2}}$ is semisimple
and there are finitely many irreducible weak $L_{\widehat{\mathfrak{g}}}(-\frac{1}{2},0)$-modules in the category $\mathcal{C}_{-\frac{1}{2}}$ up to isomorphism.
Furthermore, we consider the $\mathbb{Q}$-graded vertex operator superalgebras $(L_{\widehat{\mathfrak{g}}}(\mathcal{k},0),\omega_\xi)$
associated to a family of new Virasoro elements $\omega_\xi$, where $0<\xi<1$ is a rational number.
We apply the semisimplicity of the category $\mathcal{C}_{-\frac{1}{2}}$ to prove that
$(L_{\widehat{\mathfrak{g}}}(-\frac{1}{2},0),\omega_\xi)$
is a rational $\mathbb{Q}$-graded vertex operator superalgebra.
In \cite{DK}, the $A(V)$-theory of $\mathbb{Q}$-graded vertex operator superalgebra has been studied.
We use the $A(V)$-theory of $\mathbb{Q}$-graded vertex operator superalgebra to
determine the Zhu's algebra $A_{\omega_\xi}(L_{\widehat{\mathfrak{g}}}(-\frac{1}{2},0))$
of $(L_{\widehat{\mathfrak{g}}}(-\frac{1}{2},0),\omega_\xi)$.
We show that $A_{\omega_\xi}(L_{\widehat{\mathfrak{g}}}(-\frac{1}{2},0))$ is a finite-dimensional semisimple associative superalgebra
and $(L_{\widehat{\mathfrak{g}}}(-\frac{1}{2},0),\omega_\xi)$ is $C_2$-cofinite.
Finally we consider the case of the non-boundary admissible level $\frac{1}{2}$ to support our conjecture,
we show that the category $\mathcal{O}_{\frac{1}{2}}$ has infinitely many irreducible modules,
the category $\mathcal{C}_{\frac{1}{2}}$ is not semisimple
and $(L_{\widehat{\mathfrak{g}}}(\frac{1}{2},0),\omega_\xi)$ is not rational.

This paper is organized as follows.
In Section \ref{sec:2},
we recall some concepts and facts about vertex operator superalgebra, $A(V)$-theory
and affine Lie superalgebra $\widehat{\mathfrak{g}}=\widehat{{sl}(2|1)}$.
In Section \ref{sec:3},
we show that the category $\mathcal{O}_{-\frac{1}{2}}$ and $\mathcal{C}_{-\frac{1}{2}}$ are semisimple and have
finitely many irreducible modules.
In Section \ref{sec:4}, we determine the Zhu's algebra $A_{\omega_\xi}(L_{\widehat{\mathfrak{g}}}(-\frac{1}{2},0))$
and prove that $(L_{\widehat{\mathfrak{g}}}(-\frac{1}{2},0),\omega_\xi)$ is rational and $C_2$-cofinite.
In Section \ref{sec:5}, we show that the category $\mathcal{O}_{\frac{1}{2}}$ has infinitely many irreducible modules,
the category $\mathcal{C}_{\frac{1}{2}}$ is not semisimple
and $(L_{\widehat{\mathfrak{g}}}(\frac{1}{2},0),\omega_\xi)$ is not rational.
At last, we present our conjecture in the end of Section \ref{sec:5}.
Throughout the paper,
$\mathbb{Z}_+$, $\mathbb{Q}_+$ and $\mathbb{N}$ are the sets of
nonnegative integers, nonnegative rational numbers and positive integers respectively.

\section{Preliminaries}
\label{sec:2}
	\def\theequation{2.\arabic{equation}}
	\setcounter{equation}{0}

In this section, we recall some concepts and facts about vertex operator superalgebra, $A(V)$-theory
and affine Lie superalgebra $\widehat{{sl}(2|1)}$.

\subsection{Vertex operator superalgebras and their modules}
Let $V=V_{\bar{0}}\oplus V_{\bar{1}}$ be a vector superspace,
the elements in $V_{\bar{0}}$ (resp. $V_{\bar{1}}$) are called {\em even} (resp. {\em odd}).
For any $v\in V_{\bar{i}}$ with $i=0,1$, define $|v|=i$.
First we recall the definitions of $\mathbb{Z}$-graded (resp. $\mathbb{Q}$-graded) vertex operator superalgebra
and their various module categories following \cite{L}.

\begin{defi}\label{defivosa}
{\em A {\em vertex superalgebra} is a quadruple $(V,\textbf{1},D,Y)$,
where $V=V_{\bar{0}}\oplus V_{\bar{1}}$ is a vector superspace,
$D$ is an endomorphism of $V$, $\textbf{1}$ is a specified even vector called the {\em vacuum vector} of $V$,
and $Y$ is a linear map
\begin{equation*}
\begin{aligned}
Y(\cdot,z):&V\rightarrow (\mbox{End} V)[[z,z^{-1}]]\\
&a\mapsto Y(a,z)=\sum_{n\in\mathbb{Z}}a_n z^{-n-1} ~~(a_n\in\mbox{End} V)
\end{aligned}
\end{equation*}
such that\\
(1) For any $a,b\in V$, $a_nb=0$ for $n$ sufficiently large;\\
(2) $[D, Y(a,z)]=Y(D(a),z)=\frac{d}{dz}Y(a,z)$ for any $a\in V$;\\
(3) $Y(\textbf{1},z)=\mbox{Id}_V$;\\
(4) $Y(a,z)\textbf{1}\in V[[z]]$ and $\mbox{lim}_{z\rightarrow0}Y(a,z)\textbf{1}=a$ for any $a\in V$;\\
(5) For $\mathbb{Z}_2$-homogeneous elements $a,b\in V$, the following {\em Jacobi identity} holds:
\begin{equation*}
\begin{aligned}
  &z_0^{-1}\delta(\frac{z_1-z_2}{z_0})Y(a,z_1)Y(b,z_2)-(-1)^{|a||b|}z_0^{-1}\delta(\frac{z_2-z_1}{-z_0})Y(b,z_2)Y(a,z_1)\\
  &=z_2^{-1}\delta(\frac{z_1-z_0}{z_2})Y(Y(a,z_0)b,z_2).
 \end{aligned}
\end{equation*}

A vertex superalgebra $V$ is called a {\em $\mathbb{Z}$-graded} (resp. {\em $\mathbb{Q}$-graded}) {\em vertex operator superalgebra}
if there is an even vector $\omega$ called the {\em Virasoro element} of $V$ such that\\
(6) Set $Y(\omega,z)=\sum_{n\in\mathbb{Z}}L(n) z^{-n-2}$, for any $m,n\in\mathbb{Z}$,
$$[L(m), L(n)]=(m-n)L(m+n)+\frac{m^{3}-m}{12}\delta_{m+n,0}c,$$
where $c\in\mathbb{C}$ is called the {\em central charge};\\
(7) $L(-1)=D$;\\
(8) $V$ is $\mathbb{Z}$-graded (resp. $\mathbb{Q}$-graded) such that $V=\bigoplus_{n\in\mathbb{Z}}V_{(n)}$ (resp. $V=\bigoplus_{n\in\mathbb{Q}}V_{(n)}$),
$L(0)\mid_{V_{(n)}}=n \mbox{Id}_{V_{(n)}}$, $\mbox{dim}~V_{(n)}<\infty$ for all $n\in\mathbb{Z}$ (resp. $n\in\mathbb{Q}$) and
$V_{(n)}=0$ for $n$ sufficiently small. For $v\in V_{(n)}$,
the {\em conformal weight} of $v$ is defined to be $\mbox{wt}~v=n$.}
\end{defi}

\begin{defi}
{\em  Let $V$ be a $\mathbb{Z}$-graded (resp. $\mathbb{Q}$-graded) vertex operator superalgebra.
A {\em weak $V$-module} is a vector supersapce $M=M_{\bar{0}}\oplus M_{\bar{1}}$
equipped with a linear map
\begin{equation*}
\begin{aligned}
Y_M(\cdot,z):~&V\rightarrow (\mbox{End} M)[[z,z^{-1}]]\\
&a\mapsto Y_M(a,z)=\sum_{n\in\mathbb{Z}}a_n z^{-n-1} ~~(a_n\in\mbox{End} M)
\end{aligned}
\end{equation*}
such that\\
(1) For any $a\in V,u\in M$, $a_nu=0$ for $n$ sufficiently large;\\
(2) $Y_M(\textbf{1},z)=\mbox{Id}_M$;\\
(3) For $\mathbb{Z}_2$-homogeneous elements $a,b\in V$, the following {\em Jacobi identity} holds:
\begin{equation*}
\begin{aligned}
  &z_0^{-1}\delta(\frac{z_1-z_2}{z_0})Y_M(a,z_1)Y_M(b,z_2)-(-1)^{|a||b|}z_0^{-1}\delta(\frac{z_2-z_1}{-z_0})Y_M(b,z_2)Y_M(a,z_1)\\
  &=z_2^{-1}\delta(\frac{z_1-z_0}{z_2})Y_M(Y(a,z_0)b,z_2).
 \end{aligned}
\end{equation*}

A weak $V$-module $M$ is called a {\em $\mathbb{Z}_+$-graded} (resp. {\em $\mathbb{Q}_+$-graded}) {\em weak module}
if $M$ has a $\mathbb{Z}_+$-gradation $M=\bigoplus_{n\in\mathbb{Z}_+}M(n)$ (resp. $\mathbb{Q}_+$-gradation $M=\bigoplus_{n\in\mathbb{Q}_+}M(n)$) such that
\begin{equation*}
  a_mM(n)\subseteq M(\mbox{wt}~a+n-m-1)
\end{equation*}
for any homogeneous element $a\in V, m\in\mathbb{Z}, n\in\mathbb{Z}_+~~(\mbox{resp.}~~\mathbb{Q}_+)$.

A weak $V$-module $M$ is called an {\em ordinary module}
if $M=\bigoplus_{\lambda\in\mathbb{C}}M_{\lambda}$, where $M_{\lambda}=\{w\in M \mid L(0)w=\lambda w\}$ such that
$\mbox{dim}~M_{\lambda}<\infty$ for all $\lambda\in\mathbb{C}$ and
$M_{\lambda}=0$ for the real part of $\lambda$ sufficiently small.}
\end{defi}

\begin{defi}
{\em A $\mathbb{Z}$-graded (resp. $\mathbb{Q}$-graded) vertex operator superalgebra $V$
is called {\em rational} if any $\mathbb{Z}_+$-graded (resp. $\mathbb{Q}_+$-graded) weak $V$-module is a direct
sum of irreducible $\mathbb{Z}_+$-graded (resp. $\mathbb{Q}_+$-graded) weak $V$-modules.  }
\end{defi}

Let $(V,Y,\textbf{1},\omega)$ be a $\mathbb{Z}$-graded vertex operator superalgebra,
$h\in V_{(1)}$ a vector satisfying the following conditions:
\begin{equation}\label{eql}
\begin{aligned}
&[L(m),h_n]=-nh_{m+n}-\frac{m^2+m}{2}\delta_{m+n,0}\kappa_1\mbox{Id}_V,\\
&[h_m,h_n]=2m\delta_{m+n,0}\kappa_2\mbox{Id}_V,
\end{aligned}
\end{equation}
where $\kappa_1,\kappa_2\in\mathbb{C}$.
Assume that $h_0$ acts semisimply on $V$ and the eigenvalues of $h_0$ are rational numbers.
For a rational number $\xi$, set $\omega_\xi=\omega+\frac{\xi}{2}L(-1)h$ and $Y(\omega_\xi,z)=\sum_{n\in\mathbb{Z}}L^\prime(n) z^{-n-2}$.
Then we have
$$L^\prime(n)=L(n)-\frac{\xi}{2}(n+1)h_n~~\mbox{for~~any}~~n\in\mathbb{Z},$$
hence the component operators of $\omega_\xi$
satisfy the Virasoro relations and $L^\prime(-1)=L(-1)=D$.
Since $L(0)$ and $h_0$ are commutative,
each $V_{(n)}$ is a direct sum of eigenspaces of $h_0$.
Set $V_{(m,n)}=\{v\in V_{(m)}\mid h_0v=nv\}$ for $m\in\mathbb{Z},n\in \mathbb{Q}$.
Let
\begin{equation*}
V^\prime_{(m)}=\{v\in V\mid L^\prime(0)v=mv\}=\coprod_{s\in\mathbb{Z},t\in \mathbb{Q},s-\frac{\xi t}{2}=m}V_{(s,t)}.
\end{equation*}
Then we have the following proposition (cf. \cite{AV,DLM2,Lin}).

\begin{prop}\label{propqvosa}
Let $(V,Y,\textbf{1},\omega)$ be a $\mathbb{Z}$-graded vertex operator superalgebra
with central charge $c$. Let $h\in V_{(1)}$ be a vector satisfying the conditions (\ref{eql}),
$h_0$ acts semisimply on $V$ such that the eigenvalues of $h_0$ are rational numbers.
Suppose that $\mbox{dim}~V^\prime_{(m)}<\infty$ for any $m\in\mathbb{Q}$
and $V^\prime_{(m)}=0$ for $m$ sufficiently small.
Then $(V,Y,\textbf{1},\omega_\xi)$ is a $\mathbb{Q}$-graded vertex operator superalgebra with central charge $c-6\xi(\kappa_1+\xi\kappa_2)$.
\end{prop}

\subsection{Zhu's algebra and $C_2$-cofiniteness}
Let $V$ be a $\mathbb{Q}$-graded vertex operator superalgebra.
Define a function $\varepsilon$ for all homogeneous elements of $V$ as follows:
\begin{equation*}
\varepsilon(a)=\begin{cases}
1,& \mbox{wt}~a\in\mathbb{Z},\\
0,& \mbox{wt}~a\notin\mathbb{Z}.\\
\end{cases}
\end{equation*}
For any homogeneous element $a\in V$, we define
\begin{equation}
a\ast b=\varepsilon(a)\mbox{Res}_z\frac{(1+z)^{[{\rm wt}~a]}}{z}Y(a,z)b
\end{equation}
for any $b\in V$, where $[\cdot]$ denotes the greatest-integer function.
Then we can extend $\ast$ on $V$.
Let $O(V)$ be the subspace of $V$ linearly spanned by
\begin{equation}
\mbox{Res}_z\frac{(1+z)^{[{\rm wt}~a]}}{z^{1+\varepsilon(a)}}Y(a,z)b
\end{equation}
for any homogeneous element $a\in V$ and for any $b\in V$.
Similar to the Lemma 2.1.2 of \cite{Z}, we have
\begin{equation*}
\mbox{Res}_z\frac{(1+z)^{[{\rm wt}~a]+m}}{z^{1+\varepsilon(a)+n}}Y(a,z)b\in O(V)
\end{equation*}
for $n\geq m\geq0$.
Let $M$ be any weak $V$-module, we define
\begin{equation}\nonumber
\Omega(M)=\{u\in M\mid a_m u=0~~\mbox{for}~~a\in V, m>\mbox{wt}~a-1\},
\end{equation}
and $o$ to be the linear map from $V$ to $\mbox{End}~\Omega(M)$ such that $o(a)=\varepsilon(a)a_{[{\rm wt}~a]-1}$
for any homogeneous element $a\in V$.
The following theorem was established in \cite{DK}.
\begin{thm}\label{thmzhu}
Let $V$ be a $\mathbb{Q}$-graded vertex operator superalgebra and $M$ a weak $V$-module.\\
{\rm (a)} The subspace $O(V)$ is a two-sided ideal of $V$ with respect to the product $\ast$
and $A(V)=V/O(V)$ is an associative superalgebra with identity $\textbf{1}+O(V)$.
Moreover, $\omega+O(V)$ lies in the center of $A(V)$.\\
{\rm (b)} $\Omega(M)$ is an $A(V)$-module with $a$ acts as $o(a)$ for any $a\in V$.\\
{\rm (c)} There is an induction functor $L$ from the category of $A(V)$-modules to the category of $\mathbb{Q}_+$-graded weak $V$-modules.
Moreover, $\Omega(L(U))=U$ for any $A(V)$-module $U$.\\
{\rm (d)} $\Omega$ and $L$ are inverse bijections between the sets of irreducible modules in each category.\\
\end{thm}
\vspace{-0.5cm}

The following proposition was established in \cite{Lin}.

\begin{prop}\label{prop22}
Let $V$ be a $\mathbb{Q}$-graded vertex operator superalgebra.
If $V$ is rational as a $\mathbb{Q}$-graded vertex operator superalgebra,
then $A(V)$ is a finite-dimensional semisimple associative superalgebra.
\end{prop}

For $a\in V$, denote by $\overline{a}$ the image of $a$ under the projection of $V$ onto $A(V)$.
Similar to the Proposition 1.4.2 of \cite{FZ}, we have the following lemma.

\begin{lem}\label{lemfz}
Let $V$ be a $\mathbb{Q}$-graded vertex operator superalgebra and $I$ an ideal of $V$.
Then $(I+O(V))/O(V)$ is a two-sided ideal of $A(V)$ and $A(V/I)\cong A(V)/((I+O(V))/O(V))$.\\
\end{lem}
\vspace{-0.5cm}

Let $V$ be a $\mathbb{Q}$-graded vertex operator superalgebra
and $M$ a weak $V$-module.
Define $C_2(M)$ to be the subspace linearly spanned by $a_{-2}M$ for $a\in V_{(n)}(n\in\mathbb{Z})$
and $a_{-1}M$ for $a\in V_{(n)}(n\notin\mathbb{Z})$ (cf. \cite{AV,DLM2}).
Define bilinear products ''$\cdot$'' and ''$\circ$'' on $V$ as follows:
For $a\in V_{(m)},b\in V_{(n)}$ we define $a\cdot b=a_{-1}b$ and $a\circ b=a_{0}b$ if $m,n\in\mathbb{Z}$,
otherwise we define $a\cdot b=0$ and $a\circ b=0$.
Set $A_2(M)=M/C_2(M)$.
Similar to the Section 4.4 of \cite{Z} (see also the Proposition 3.10 of \cite{DLM2}) we have the following proposition.

\begin{prop}
The subspace $C_2(V)$ is a two-sided ideal for both $(V,\cdot)$ and $(V,\circ)$.
Moreover, $(A_2(V),\cdot,\circ)$ is a commutative Poisson superalgebra such that $(A_2(V),\cdot)$ is a commutative associative superalgebra
with identity $\textbf{1}+C_2(V)$ and $(A_2(V),\circ)$ is a Lie superalgebra with
$(a\cdot b)\circ c=a\cdot(b\circ c)+(-1)^{|b||c|}(a\circ c)\cdot b$ for any $a,b,c\in A_2(V)$.
\end{prop}
\begin{defi}
{\em A $\mathbb{Q}$-graded vertex operator superalgebra $V$ is called {\em $C_2$-cofinite} if $A_2(V)$ is finite-dimensional.}
\end{defi}
\subsection{Affine Lie superalgebra $\widehat{sl(2|1)}$}

Let $\mathfrak{g}={sl}(2|1)$ be the finite dimensional simple Lie superalgebra with the basis
\begin{equation*}
\begin{aligned}
&h_1=\begin{pmatrix} 1 & 0 & 0\\ 0 & 0 & 0\\0 & 0 & 1 \end{pmatrix},
~~~~~~~e_{12}=\begin{pmatrix} 0 & 1 & 0\\ 0 & 0 & 0\\0 & 0 & 0 \end{pmatrix},
~~e_1=\begin{pmatrix} 0 & 0 & 1\\ 0 & 0 & 0\\0 & 0 & 0 \end{pmatrix},
~~e_2=\begin{pmatrix} 0 & 0 & 0\\ 0 & 0 & 0\\0 & 1 & 0 \end{pmatrix},\\
&h_2=\begin{pmatrix} 0 & 0 & 0\\ 0 & -1 & 0\\0 & 0 & -1 \end{pmatrix},
~~f_{12}=\begin{pmatrix} 0 & 0 & 0\\ 1 & 0 & 0\\0 & 0 & 0 \end{pmatrix},
~~f_1=\begin{pmatrix} 0 & 0 & 0\\ 0 & 0 & 0\\1 & 0 & 0 \end{pmatrix},
~~f_2=\begin{pmatrix} 0 & 0 & 0\\ 0 & 0 & -1\\0 & 0 & 0 \end{pmatrix},
\end{aligned}
\end{equation*}
where the even part $\mathfrak{g}_0=\mbox{span}_{\mathbb{C}}\{h_1,h_2,e_{12},f_{12}\}$ and
the odd part $\mathfrak{g}_1=\mbox{span}_{\mathbb{C}}\{e_1,e_2,f_1,f_2\}$, set $h_+=h_1+h_2,h_-=h_1-h_2$.
Let $\mathfrak{h}=\mathbb{C}h_1\oplus\mathbb{C}h_2$ be the Cartan subalgebra of $\mathfrak{g}$
and $\mathfrak{g}=\mathfrak{n}_-\oplus\mathfrak{h}\oplus\mathfrak{n}_+$ a triangular decomposition of $\mathfrak{g}$,
where $\mathfrak{n}_-=\mbox{span}_{\mathbb{C}}\{f_1,f_2,f_{12}\}$ and $\mathfrak{n}_+=\mbox{span}_{\mathbb{C}}\{e_1,e_2,e_{12}\}$.
We fix the set of simple roots $\Pi=\{\alpha_1,\alpha_2\}$ such that $\alpha_1,\alpha_2$ are odd roots,
and the set of simple coroots $\Pi^{\vee}=\{h_1,h_2\}$, then the corresponding Cartan matrix $A=\begin{pmatrix} 0 & 1\\ 1 & 0  \end{pmatrix}$.
Denote by $\Delta$ the root system of $\mathfrak{g}$ and $\theta$ the highest root.
Let $(\cdot,\cdot)$ be the invariant nondegenerate supersymmetric bilinear form on $\mathfrak{g}$
such that non-trivial products are given by $(e_{12},f_{12})=(h_1,h_2)=(e_1,f_1)=(e_2,f_2)=1$.

Let $\tilde{\mathfrak{g}}=\mathfrak{g}\otimes\mathbb{C}[t,t^{-1}]\oplus\mathbb{C}k$ be the affine Lie superalgebra of ${sl}(2|1)$,
we identify $\mathfrak{g}$ with $\mathfrak{g}\otimes t^0$ and
set $a(n)=a\otimes t^n$ for $a\in\mathfrak{g}$, $n\in\mathbb{Z}$ for convenience.
Define subalgebras
\begin{equation}
\begin{aligned}
&N_+=\mathfrak{n}_+\oplus\mathfrak{g}\otimes t\mathbb{C}[t],~
N_-=\mathfrak{n}_-\oplus\mathfrak{g}\otimes t^{-1}\mathbb{C}[t^{-1}],\\
&B=N_+\oplus\mathfrak{h}\oplus\mathbb{C}k,~P=\mathfrak{g}\otimes\mathbb{C}[t]\oplus\mathbb{C}k.
\end{aligned}
\end{equation}
Let $\hat{\mathfrak{g}}=\tilde{\mathfrak{g}}\oplus\mathbb{C}d$ be the extended affine Lie superalgebra
and $\hat{\mathfrak{h}}=\mathfrak{h}\oplus\mathbb{C}k\oplus\mathbb{C}d$ the Cartan subalgebra of $\hat{\mathfrak{g}}$.
Let $\delta$ be the linear function on $\hat{\mathfrak{h}}$ defined by $\delta|_{\mathfrak{h}\oplus\mathbb{C}k}=0, \delta(d)=1$,
$\widehat{\Pi}=\{\alpha_0=\delta-\theta,\alpha_1,\alpha_2\}$ the set of simple roots.
Let $\alpha^{\vee}=\alpha$ for simple isotropic root $\alpha\in\widehat{\Pi}$, and $\alpha^{\vee}=2\alpha/(\alpha,\alpha)$
for simple non-isotropic root $\alpha\in\widehat{\Pi}$.
Define the fundamental weights $\Lambda_i\in \hat{\mathfrak{h}}^* (i=0,1,2)$ by $(\Lambda_i,\alpha_j^{\vee})=\delta_{ij}, \Lambda_i(d)=0$ for any $i,j\in\{0,1,2\}$.
For any $\hat{\lambda}\in \hat{\mathfrak{h}}^*$,
denote by $M(\hat{\lambda})$(resp. $L(\hat{\lambda})$) the Verma (resp. irreducible highest weight) $\hat{\mathfrak{g}}$-module.
Note that for $\widehat{{sl}(2|1)}$ all admissible weights $\hat{\lambda}\in \hat{\mathfrak{h}}^*$ are principal (cf. \cite{KW14,KRW}),
then from the Proposition 3.14 of \cite{KW14}, we have that
the level $\mathcal{k}=\hat{\lambda}(k)$ is admissible if and only if $\mathcal{k}+1=\frac{m+1}{M}$
and the admissible level $\mathcal{k}$ is boundary if and only if $\mathcal{k}+1=\frac{1}{M}$,
where $m\in\mathbb{Z}_+, M\in\mathbb{N}, (M,m+1)=1$.
For any boundary admissible level $\mathcal{k}$, we also have that
all admissible weights of level $\mathcal{k}$ are
$$\Lambda_{k_1,k_2}^{(1)}=-k_0(\mathcal{k}+1)\Lambda_0-k_1(\mathcal{k}+1)\Lambda_1-k_2(\mathcal{k}+1)\Lambda_2$$
where $M=k_0+k_1+k_2+1,k_0,k_1,k_2\in\mathbb{Z}_+$, and
$$\Lambda_{k_1,k_2}^{(2)}=(k_0(\mathcal{k}+1)-2)\Lambda_0+k_1(\mathcal{k}+1)\Lambda_1+k_2(\mathcal{k}+1)\Lambda_2$$
where $M=k_0+k_1+k_2-1,k_0,k_1,k_2\in\mathbb{Z}_+,1\leq k_1,k_2\leq M-1,k_1+k_2\leq M$.
For any admissible level $\mathcal{k}$, the weight $\mathcal{k}\Lambda_0$ is admissible.

It is clear that any $\hat{\mathfrak{g}}$-module is also a $\tilde{\mathfrak{g}}$-module,
denote by $L(\mathcal{k},\lambda)$ the $\tilde{\mathfrak{g}}$-module $L(\hat{\lambda})$, $M(\mathcal{k},\lambda)$ the $\tilde{\mathfrak{g}}$-module $M(\hat{\lambda})$, where $\mathcal{k}=\hat{\lambda}(k),\lambda=\hat{\lambda}|_{\mathfrak{h}}$.
For convenience, we set $\Lambda_i=\Lambda_i|_{\mathfrak{h}}(i=1,2)$.
For any irreducible highest weight $\hat{\mathfrak{g}}$-module $L(\hat{\lambda})$,
$L(\mathcal{k},\lambda)$ is also irreducible as $\tilde{\mathfrak{g}}$-module.
Conversely, for any restricted $\tilde{\mathfrak{g}}$-module $M$ of level $\mathcal{k}\neq-1$,
by the Sugawara construction we can extend $M$ to a $\hat{\mathfrak{g}}$-module by letting $d$ acts on $M$ as $-L(0)$.
In this paper we shall consider any restricted $\tilde{\mathfrak{g}}$-module as a $\hat{\mathfrak{g}}$-module in this way.
Let $\mathcal{k}\in\mathbb{C}$ and $U$ a $\mathfrak{h}$-module,
$U$ can be regarded as a $B$-module by letting $N_+$ acts trivially and $k$ acts as $\mathcal{k}$,
let $M(\mathcal{k},U)=U(\tilde{\mathfrak{g}})\otimes_{U(B)}U$.
If $U=\mathbb{C}$ is a one-dimensional $\mathfrak{h}$-module on which $h\in\mathfrak{h}$ acts as a fixed complex number $\lambda(h)$ with $\lambda\in\mathfrak{h}^*$,
the corresponding module is a Verma module denoted by $M(\mathcal{k},\lambda)$.
Then $M(\mathcal{k},\lambda)$ has an unique maximal submodule and $L(\mathcal{k},\lambda)$ is isomorphic to the corresponding irreducible quotient.
Similarly we can define the generalized Verma $\tilde{\mathfrak{g}}$-module $V(\mathcal{k},U)=U(\tilde{\mathfrak{g}})\otimes_{U(P)}U$
for any $\mathfrak{g}$-module $U$ which can be extend to a $P$-module by letting $\mathfrak{g}\otimes t\mathbb{C}[t]$
acts trivially and $k$ acts as $\mathcal{k}$.
Note that if $U=\mathbb{C}$ is the trivial $\mathfrak{g}$-module,
then $V(\mathcal{k},\mathbb{C})$ is a quotient of Verma module $M(\mathcal{k},0)$
and $L(\mathcal{k},0)$ is isomorphic to the irreducible quotient of $V(\mathcal{k},\mathbb{C})$.
Let $r_\alpha$ be the reflection with respect to an even root $\alpha$ and $\rho$ be the sum of fundamental weights,
set $r_{\alpha}.\lambda=r_{\alpha}(\lambda+\rho)-\rho$.
\begin{prop}\label{lemsv1}
Let $\mathcal{k}=\frac{m+1}{M}-1$ be an admissible level.
Then we have
$$L(\mathcal{k},0)=V(\mathcal{k},\mathbb{C})\Big/U(\tilde{\mathfrak{g}})v,$$
where $v\in V(\mathcal{k},\mathbb{C})$ is a singular vector of weight $r_{\alpha_0+(M-1)\delta}.\mathcal{k}\Lambda_0$.
\end{prop}
\begin{proof}
From the Lemma 7.1 of \cite{KW04}, $r_{\alpha_0+(M-1)\delta}.\mathcal{k}\Lambda_0,\mathcal{k}\Lambda_0-\alpha_1$ and $\mathcal{k}\Lambda_0-\alpha_2$
are singular weights of $M(\mathcal{k},0)$, let $v_0,v_1,v_2$ be the corresponding singular vectors respectively.
It is clear that the image of $v_1$ and $v_2$ in $V(\mathcal{k},\mathbb{C})$ are zero, let $v$ be the image of $v_0$ in $V(\mathcal{k},\mathbb{C})$.
The character formula for all admissible modules over $\widehat{\mathfrak{g}}$ is obtained in \cite{GK15}.
The definition of the map $F$ in the section 7.7 of \cite{GK15} can be extended to the highest weight modules,
then we have $M(\mathcal{k},0)\Big/(U(\hat{\mathfrak{g}})v_0+U(\hat{\mathfrak{g}})v_1+U(\hat{\mathfrak{g}})v_2)$
is relatively integrable in the sense of the section 10.1 of \cite{GK15}.
Similar to the proof of the Corollary 11.2.4 in \cite{GK15}, we can obtain the character formula of $M(\mathcal{k},0)\Big/(U(\hat{\mathfrak{g}})v_0+U(\hat{\mathfrak{g}})v_1+U(\hat{\mathfrak{g}})v_2)$,
 which coincide with the character formula of $L(\mathcal{k}\Lambda_0)$, hence $L(\mathcal{k}\Lambda_0)=M(\mathcal{k},0)\Big/(U(\hat{\mathfrak{g}})v_0+U(\hat{\mathfrak{g}})v_1+U(\hat{\mathfrak{g}})v_2)$.
From the Theorem 8 of \cite{GK07}, $V(\mathcal{k},\mathbb{C})$ is not simple, then $v\ne 0$,
hence $L(\mathcal{k}\Lambda_0)=V(\mathcal{k},\mathbb{C})\big/U(\tilde{\mathfrak{g}})v$.
Then as $\tilde{\mathfrak{g}}$-module, $L(\mathcal{k},0)=V(\mathcal{k},\mathbb{C})\big/U(\tilde{\mathfrak{g}})v$.
\end{proof}
\begin{rem}
{\rm The  Proposition \ref{lemsv1} has also been proved in the Proposition 11.1 of \cite{GK25} in a different way.}
\end{rem}

\subsection{Vertex operator superalgebra $V_{\widehat{\mathfrak{g}}}(\mathcal{k},0)$ and $L_{\widehat{\mathfrak{g}}}(\mathcal{k},0)$}

We know that for $\mathcal{k}\neq -1$, both
$V(\mathcal{k},\mathbb{C})$ and $L(\mathcal{k},0)$ have natural
$\mathbb{Z}$-graded vertex operator superalgebra structures
and any $\tilde{\mathfrak{g}}$-module $M(\mathcal{k},U)$ is a weak module for $V(\mathcal{k},\mathbb{C})$,
we denote by $V_{\widehat{\mathfrak{g}}}(\mathcal{k},0)$ the universal affine vertex operator superalgebra $V(\mathcal{k},\mathbb{C})$
and $L_{\widehat{\mathfrak{g}}}(\mathcal{k},0)$ the simple affine vertex operator superalgebra $L(\mathcal{k},0)$.
Let $\mathbf{1}$ be the vacuum vector of $V_{\widehat{\mathfrak{g}}}(\mathcal{k},0)$, similar to the Theorem 3.1.1 of \cite{FZ}, we have the following lemma.
\begin{lem}\label{lemzhu}
The associative superalgebra $A(V_{\widehat{\mathfrak{g}}}(\mathcal{k},0))$ is canonically isomorphic to $U(\mathfrak{g})$.
The isomorphism is given by
\begin{equation*}
\begin{aligned}
F:~~~~~~~~~~~~~~~~~~~~~~A(V_{\widehat{\mathfrak{g}}}(\mathcal{k},0))&\rightarrow U(\mathfrak{g})\\
\overline{a_1(-n_1-1)\cdots a_m(-n_m-1)\mathbf{1}}&\mapsto (-1)^{\sum_{i<j}|a_i||a_j|+\sum_i n_i}a_m\cdots a_1,
\end{aligned}
\end{equation*}
where $a_1,\cdots,a_m\in\mathfrak{g},n_1,\cdots,n_m\in\mathbb{Z}_+$.
\end{lem}
 From the Lemma \ref{lemfz}, Proposition \ref{lemsv1} and Lemma \ref{lemzhu}, we can characterize the Zhu's algebra $A(L_{\widehat{\mathfrak{g}}}(\mathcal{k},0))$ of affine vertex operator superalgebra $L_{\widehat{\mathfrak{g}}}(\mathcal{k},0)$.

\begin{prop}
Let $v$ be the singular vector of $V(\mathcal{k},\mathbb{C})$ which generates the maximal submodule.
Then $$A(L_{\widehat{\mathfrak{g}}}(\mathcal{k},0))\cong U(\mathfrak{g})\big/\langle F(\overline{v})\rangle,$$
where $\langle F(\overline{v})\rangle$ is the two-sided ideal of $U(\mathfrak{g})$ generated by $F(\overline{v})$.
Let $U$ be a $\mathfrak{g}$-module, then $U$ is an $A(L_{\widehat{\mathfrak{g}}}(\mathcal{k},0))$-module if and only if $F(\overline{v})U=0$.
\end{prop}

For the $\mathbb{Z}$-graded vertex operator superalgebra $L_{\widehat{\mathfrak{g}}}(\mathcal{k},0)$,
by the Sugawara construction,
\begin{equation*}
\begin{aligned}
L(n)=\frac{1}{2(\mathcal{k}+1)}\sum_{m\in\mathbb{Z}}&(:h_1(m)h_2(n-m):+:h_2(m)h_1(n-m):+:e_{12}(m)f_{12}(n-m):\\&+:f_{12}(m)e_{12}(n-m):
-:e_1(m)f_1(n-m):+:f_1(m)e_1(n-m):\\&-:e_2(m)f_2(n-m):+:f_2(m)e_2(n-m):),
\end{aligned}
\end{equation*}
we have
$[L(m),h_+(n)]=-nh_+(m+n)$, then $h_+$ satisfies conditions (\ref{eql})
with $\kappa_1=0, \kappa_2=\mathcal{k}$.
Hence $L_{\widehat{\mathfrak{g}}}(\mathcal{k},0)$ is $\mathbb{Q}$-graded by weights with respect to $L^\prime(0)$.
From the construction of $V_{\widehat{\mathfrak{g}}}(\mathcal{k},0)$,
for any $m\in\mathbb{Z}_+$, if $V_{\widehat{\mathfrak{g}}}(\mathcal{k},0)_{(m,n)}\ne 0$,
we have $-2m\leq n\leq 2m$.
Since $L_{\widehat{\mathfrak{g}}}(\mathcal{k},0)$ is a quotient of $V_{\widehat{\mathfrak{g}}}(\mathcal{k},0)$,
for any $m\in\mathbb{Z}_+$, if $L_{\widehat{\mathfrak{g}}}(\mathcal{k},0)_{(m,n)}\ne 0$,
we also have $-2m\leq n\leq 2m$.
Take $\xi\in\mathbb{Q}$ such that
$0<\xi<1$. For any $m\in\mathbb{Q}$,
since $$L_{\widehat{\mathfrak{g}}}(\mathcal{k},0)^\prime_{(m)}=\coprod_{s\in\mathbb{Z},t\in \mathbb{Q},s-\frac{\xi t}{2}=m}L_{\widehat{\mathfrak{g}}}(\mathcal{k},0)_{(s,t)},$$
we have $\frac{m}{1+\xi}\leq s\leq\frac{m}{1-\xi}$
for $L_{\widehat{\mathfrak{g}}}(\mathcal{k},0)_{(s,t)}\ne0$,
then $\mbox{dim}~L_{\widehat{\mathfrak{g}}}(\mathcal{k},0)^\prime_{(m)}<\infty$.
Note that $\frac{m}{1+\xi}\leq \frac{m}{1-\xi}$,
we have $L_{\widehat{\mathfrak{g}}}(\mathcal{k},0)^\prime_{(m)}=0$ for $m<0$.
Therefore, by the Proposition \ref{propqvosa},
$(L_{\widehat{\mathfrak{g}}}(\mathcal{k},0),\omega_\xi)$ is a $\mathbb{Q}$-graded vertex operator superalgebra
with central charge $c-6\xi^2\mathcal{k}$.
Similarly $(V_{\widehat{\mathfrak{g}}}(\mathcal{k},0),\omega_\xi)$ is also a $\mathbb{Q}$-graded vertex operator superalgebra
with central charge $c-6\xi^2\mathcal{k}$.
\section{Category $\mathcal{O}_{-\frac{1}{2}}$ and $\mathcal{C}_{-\frac{1}{2}}$}
\label{sec:3}
	\def\theequation{3.\arabic{equation}}
	\setcounter{equation}{0}

In this section,
we study the category $\mathcal{O}_{-\frac{1}{2}}$
of weak $L_{\widehat{\mathfrak{g}}}(-\frac{1}{2},0)$-modules in the category $\mathcal{O}$, and the category $\mathcal{C}_{-\frac{1}{2}}$ of weak $L_{\widehat{\mathfrak{g}}}(-\frac{1}{2},0)$-modules in the category $\mathcal{C}$,
where the category $\mathcal{C}$ is the category of $\widehat{\mathfrak{g}}$-modules
on which the positive part of $\widehat{\mathfrak{g}}$ acts locally nilpotently.
First we show that $\mathcal{O}_{-\frac{1}{2}}$ is semisimple and has finitely many irreducible modules,
hence any ordinary
$L_{\widehat{\mathfrak{g}}}(-\frac{1}{2},0)$-module is completely reducible and
there are finitely many irreducible ordinary $L_{\widehat{\mathfrak{g}}}(-\frac{1}{2},0)$-modules.
Finally, we prove that $\mathcal{C}_{-\frac{1}{2}}$ is semisimple and has finitely many irreducible modules.

\subsection{Category $\mathcal{O}_{-\frac{1}{2}}$}

We say that a $\widehat{\mathfrak{g}}$-module $M$ is from the category $\mathcal{O}$ if
the Cartan subalgebra $\widehat{\mathfrak{h}}$ acts semisimply on $M$ with finite dimensional weight spaces
and finite number of maximal weights.
Denote by $\mathcal{O}_{\mathcal{k}}$ the full subcategory of
weak $L_{\widehat{\mathfrak{g}}}(\mathcal{k},0)$-module category such that $M$ is an object of $\mathcal{O}_{\mathcal{k}}$
if and only if $M\in\mathcal{O}$ as $\widehat{\mathfrak{g}}$-module.

Similarly we denote by $\mathcal{O}_{\mathfrak{g}}$ the full subcategory of
$\mathfrak{g}$-module category such that $M$ is an object in $\mathcal{O}_{\mathfrak{g}}$ if
the Cartan subalgebra $\mathfrak{h}$ acts semisimply on $M$ with finite dimensional weight spaces
and finite number of maximal weights.
Then from the Theorem \ref{thmzhu},
there is a one-to-one correspondence between the irreducible weak $L_{\widehat{\mathfrak{g}}}(\mathcal{k},0)$-modules in the category $\mathcal{O}_{\mathcal{k}}$
and the irreducible $A(L_{\widehat{\mathfrak{g}}}(\mathcal{k},0))$-modules from the category $\mathcal{O}_{\mathfrak{g}}$.

Denote by $_L$ the adjoint action of $\mathfrak{g}$ on $U(\mathfrak{g})$ defined by $x_Lu=[x,u]$ for $x\in \mathfrak{g}$ and $u\in U(\mathfrak{g})$,
we can extend the action $_L$ of $\mathfrak{g}$ to an action of $U(\mathfrak{g})$.
Let $v$ be the singular vector of $V(\mathcal{k},\mathbb{C})$ which generates the maximal submodule,
$U(\mathfrak{g})v$ the $\mathfrak{g}$-module generated by $v$ where $x\in\mathfrak{g}$ acts on $v$ by $x(0)$.
Recall that $F$ is an isomorphism from $A(V_{\widehat{\mathfrak{g}}}(\mathcal{k},0))$ to $U(\mathfrak{g})$,
let $R=F(\overline{U(\mathfrak{g})v})$,
since each conformal weight space of $V(\mathcal{k},\mathbb{C})$ is finite-dimensional and all elements of $U(\mathfrak{g})v$
have the same conformal weight as that of $v$, then $R$ is finite-dimensional.
Note that for any $x\in\mathfrak{g}$ and $w\in V(\mathcal{k},\mathbb{C})$ we have
$F(\overline{x(0)w})=[x, F(\overline{w})]$ (cf. \cite{K}),
thus $R$ is a finite-dimensional highest weight $\mathfrak{g}$-module.
Let $R_0$ be the zero-weight subspace of $R$.
Denote by $S(\mathfrak{h})$ the symmetric algebra on $\mathfrak{h}$,
for any $p\in S(\mathfrak{h})$ and $\lambda\in\mathfrak{h}^*$,
define $p(\lambda)\in\mathbb{C}$ with $p(h_1,h_2).v_{\lambda}=p(\lambda)v_{\lambda}$.
Let $r\in R_0$, it is clear that there exists an unique polynomial $p_r\in S(\mathfrak{h})$ such that $rv_{\lambda}=p_r(\lambda)v_{\lambda}$ for any $\lambda\in\mathfrak{h}^*$.
Set $P_0=\{p_r\mid r\in R_0\}$.
Similar to the Lemma 3.4.3 of \cite{AM} and
the Proposition 13 of \cite{P}, we have the following Proposition.
\begin{prop}\label{propavvv}
Let $L(\lambda)$ be an irreducible highest weight $\mathfrak{g}$-module with the highest weight vector $v_{\lambda}$ for $\lambda\in\mathfrak{h}^*$.
The following statements are equivalent:\\
{\rm (1)} $L(\lambda)$ is an $A(L_{\widehat{\mathfrak{g}}}(\mathcal{k},0))$-module;\\
{\rm (2)} $RL(\lambda)=0$;\\
{\rm (3)} $R_0v_{\lambda}=0$;\\
{\rm (4)} $p(\lambda)=0$ for all $p\in P_0$.
\end{prop}

For boundary admissible level $-\frac{1}{2}$,
we obtain the following singular vector of $\widetilde{\mathfrak{g}}$-module $V(-\frac{1}{2},\mathbb{C})$.

\begin{lem}\label{lemsv00}
$v_1=(2e_1(-1)e_2(-1)+2h_-(-1)e_{12}(-1)-e_{12}(-2))\mathbf{1}$ is a singular vector of $\widetilde{\mathfrak{g}}$-module $V(-\frac{1}{2},\mathbb{C})$
with weight $r_{\alpha_0+\delta}.(-\frac{1}{2}\Lambda_0)$.
\end{lem}
\begin{proof}
It sufficients to show $e_1(0).v_1=e_2(0).v_1=f_{12}(1).v_1=0$.
Since
\begin{align*}
e_1(0).v_1&=2e_1(0)e_1(-1)e_2(-1)\mathbf{1}+2e_1(0)h_-(-1)e_{12}(-1)\mathbf{1}-e_1(0)e_{12}(-2)\mathbf{1}\\&=
-2e_1(-1)e_{12}(-1)\mathbf{1}+2e_1(-1)e_{12}(-1)\mathbf{1}=0,\\
e_2(0).v_1&=2e_2(0)e_1(-1)e_2(-1)\mathbf{1}+2e_2(0)h_-(-1)e_{12}(-1)\mathbf{1}-e_2(0)e_{12}(-2)\mathbf{1}\\&=
2e_{12}(-1)e_2(-1)\mathbf{1}-2e_2(-1)e_{12}(-1)\mathbf{1}=0,\\
f_{12}(1).v_1&=2f_{12}(1)e_1(-1)e_2(-1)\mathbf{1}+2f_{12}(1)h_-(-1)e_{12}(-1)\mathbf{1}-f_{12}(1)e_{12}(-2)\mathbf{1}\\&=
-2f_2(0)e_2(-1)\mathbf{1}+2h_-(-1)(-h_+(0)-\frac{1}{2})\mathbf{1}+h_+(-1)\mathbf{1}\\&=
-2h_2(-1)\mathbf{1}-h_-(-1)\mathbf{1}+h_+(-1)\mathbf{1}=0,
\end{align*}
then $v_1$ is a singular vector of $\widetilde{\mathfrak{g}}$-module $V(-\frac{1}{2},\mathbb{C})$.
Since $$r_{\alpha_0+\delta}.(-\frac{1}{2}\Lambda_0)=-\frac{1}{2}\Lambda_0-(-\frac{1}{2}\Lambda_0+\rho,(\alpha_0+\delta)^{\vee})(\alpha_0+\delta)=
-\frac{1}{2}\Lambda_0-2\delta +\alpha_1+\alpha_2,$$
it shows that the weight of $v_1$ is $r_{\alpha_0+\delta}.(-\frac{1}{2}\Lambda_0)$.
\end{proof}
\begin{rem}
{\rm We can also obtain the singular vector $v_1$ from \cite{BT} or \cite{S}, which generalized the Malikov-Feigin-Fuchs construction of singular vectors
 of Verma $\widehat{{sl}_2}$-module (cf. \cite{MFF}) to that of $\widehat{{sl}(2|1)}$.}
\end{rem}

From the Lemma \ref{lemsv00}, we have $F(\overline{v_1})=-2e_2e_1+2h_-e_{12}+e_{12}$,
thus we have $$A(L_{\widehat{\mathfrak{g}}}(-\frac{1}{2},0))\cong U(\mathfrak{g})\Big/\langle -2e_2e_1+2h_-e_{12}+e_{12}\rangle.$$

\begin{lem}
$P_0=\mbox{span}_{\mathbb{C}}\{p_1(h_1,h_2),p_2(h_1,h_2)\}$, where $$p_1(h_1,h_2)=h_1(2h_1-4h_2+1),~p_2(h_1,h_2)=h_2(2h_2-4h_1+1).$$
\end{lem}
\begin{proof}
It is clear that $(f_1f_2)_LF(\overline{v_1}),(f_2f_1)_LF(\overline{v_1})\in R_0$.
We have
\begin{align*}
(f_1f_2)_LF(\overline{v_1})&=[f_1,[f_2,-2e_2e_1+2h_-e_{12}+e_{12}]]=[f_1,2h_1e_1-4h_2e_1+2f_2e_{12}+e_1]\\
&=2h_1^2-4f_1e_1-4h_2h_1-2f_{12}e_{12}-2f_2e_2+h_1\\
&\equiv h_1(2h_1-4h_2+1) ~(\mbox{mod}~U(\mathfrak{g})\mathfrak{n}_+),\\
(f_2f_1)_LF(\overline{v_1})&=[f_2,[f_1,-2e_2e_1+2h_-e_{12}+e_{12}]]=[f_2,-2h_2e_2+4h_1e_2-2f_1e_{12}-e_2]\\
&=-2h_2^2+4f_2e_2+4h_1h_2+2f_{12}e_{12}+2f_1e_1-h_2\\
&\equiv -h_2(2h_2-4h_1+1) ~(\mbox{mod}~U(\mathfrak{g})\mathfrak{n}_+),
\end{align*}
hence $p_1(h_1,h_2)=h_1(2h_1-4h_2+1),p_2(h_1,h_2)=h_2(2h_2-4h_1+1)\in P_0$.
Since $\mbox{dim}~P_0\leq\mbox{dim}~R_0\leq 2$, it implies that $P_0=\mbox{span}_{\mathbb{C}}\{p_1(h_1,h_2),p_2(h_1,h_2)\}$.
\end{proof}

\begin{prop}\label{prop234}
The set $\{L(\lambda)|\lambda=0,-\frac{1}{2}\Lambda_1,-\frac{1}{2}\Lambda_2,\frac{1}{2}\Lambda_1+\frac{1}{2}\Lambda_2\}$
provides the complete list of irreducible $A(L_{\widehat{\mathfrak{g}}}(-\frac{1}{2},0))$-modules from the category $\mathcal{O}_{\mathfrak{g}}$.
\end{prop}
\begin{proof}
Let $\lambda\in\mathfrak{h}^*$
such that $p(\lambda)=0$ for all $p\in P_0$.
Then $p_1(\lambda)=0$ and $p_2(\lambda)=0$,
i.e., $\lambda(h_1)(2\lambda(h_1)-4\lambda(h_2)+1)=0$ and $\lambda(h_2)(2\lambda(h_2)-4\lambda(h_1)+1)=0$.
The solutions $(\lambda(h_1),\lambda(h_2))$ are $(0,0),(-\frac{1}{2},0),(0,-\frac{1}{2}),(\frac{1}{2},\frac{1}{2})$,
that is, $\lambda=0,-\frac{1}{2}\Lambda_1,-\frac{1}{2}\Lambda_2$ or $\frac{1}{2}\Lambda_1+\frac{1}{2}\Lambda_2$.
Thus from the Proposition \ref{propavvv}, we obtain that
$L(0),L(-\frac{1}{2}\Lambda_1),L(-\frac{1}{2}\Lambda_2)$ and $L(\frac{1}{2}\Lambda_1+\frac{1}{2}\Lambda_2)$
are all irreducible $A(L_{\widehat{\mathfrak{g}}}(-\frac{1}{2},0))$-modules from the category $\mathcal{O}_{\mathfrak{g}}$.
\end{proof}

Thus we obtain all irreducible weak $L_{\widehat{\mathfrak{g}}}(-\frac{1}{2},0)$-modules in the category $\mathcal{O}_{-\frac{1}{2}}$.

\begin{prop}
The set $\{L(-\frac{1}{2},\lambda)|\lambda=0,-\frac{1}{2}\Lambda_1,-\frac{1}{2}\Lambda_2,\frac{1}{2}\Lambda_1+\frac{1}{2}\Lambda_2\}$
provides the complete list of irreducible modules in the category $\mathcal{O}_{-\frac{1}{2}}$.
\end{prop}

Note that the set $\{L(-\frac{1}{2},\lambda)|\lambda=0,-\frac{1}{2}\Lambda_1,-\frac{1}{2}\Lambda_2,\frac{1}{2}\Lambda_1+\frac{1}{2}\Lambda_2\}$
is exactly the set of admissible $\widehat{\mathfrak{g}}$-modules of level $-\frac{1}{2}$,
we obtain the following theorem.

\begin{thm}\label{thm1}
The category $\mathcal{O}_{-\frac{1}{2}}$ is semisimple.
\end{thm}
\begin{proof}
Let $\hat{\lambda}_{i}~(i=1,2,3,4)$ be the corresponding admissible weights for the admissible $\widehat{\mathfrak{g}}$-modules of level $-\frac{1}{2}$.
For any distinct admissible weights $\hat{\lambda}_{i},\hat{\lambda}_{j}$, it is clear that $\hat{\lambda}_{i}\nleq\hat{\lambda}_{j}$,
hence $L(\hat{\lambda}_{i})$ is not isomorphic to any subquotient module of $M(\hat{\lambda}_{j})$.
Then if $M\in\mathcal{O}_{-\frac{1}{2}}$ is a highest weight module, we have that $M$ is irreducible.
Similar to the proof of the Theorem 5.5 in \cite{AKMPP}, we have $\mbox{Ext}_{\mathcal{O}_{-\frac{1}{2}}}^1(L,L^\prime)=0$
for any distinct irreducible modules $L$ and $L^\prime$ of the category $\mathcal{O}_{-\frac{1}{2}}$.
From the Lemma 4.1 of \cite{AMP}, we have $\mbox{Ext}_{\mathcal{O}_{-\frac{1}{2}}}^1(L,L)=0$ for any irreducible module $L$ of the category $\mathcal{O}_{-\frac{1}{2}}$.
Since $\mathcal{O}_{-\frac{1}{2}}$ has finitely many irreducible modules, from the Lemma 1.3.1 of \cite{GK}, the category $\mathcal{O}_{-\frac{1}{2}}$ is semisimple.
\end{proof}

Now we can obtain all irreducible ordinary $L_{\widehat{\mathfrak{g}}}(-\frac{1}{2},0)$-modules from the irreducible modules in $\mathcal{O}_{-\frac{1}{2}}$.

\begin{prop}\label{thmord}
All irreducible ordinary $L_{\widehat{\mathfrak{g}}}(-\frac{1}{2},0)$-modules are $L(-\frac{1}{2},0),L(-\frac{1}{2},\frac{1}{2}\Lambda_1+\frac{1}{2}\Lambda_2)$.
Moreover, the category of ordinary $L_{\widehat{\mathfrak{g}}}(-\frac{1}{2},0)$-modules is semisimple.
\end{prop}
\begin{proof}
Let $M$ be an irreducible ordinary $L_{\widehat{\mathfrak{g}}}(-\frac{1}{2},0)$-module, then
$M=\bigoplus_{h\in\mathbb{Z}_+}M_{h+\alpha}$ such that $\alpha\in\mathbb{C}$ and
$\mbox{dim}~M_{h+\alpha}<\infty$ for any $h\in\mathbb{Z}_+$.
For any $w\in M_{h+\alpha}, h\in\mathbb{Z}_+$,
since $M$ is a $\tilde{\mathfrak{g}}$-module
and $\mbox{wt}~aw=\mbox{wt}~w$ for any $a \in\mathfrak{g}$,
we have $M_{h+\alpha}$ is a finite dimensional $\mathfrak{g}$-module.
From the Proposition 5.4 of \cite{DLM3},
we have that $\Omega(M)=M_{\alpha}$ is an irreducible $A(L_{\widehat{\mathfrak{g}}}(-\frac{1}{2},0))$-module.
By the Proposition 1.39 of \cite{CW}, we have that $M_{\alpha}$ lies in $\mathcal{O}_{\mathfrak{g}}$.
Then from the Theorem 14.1.1 of \cite{M}, the Lemma 1.4 of \cite{KW01} and Proposition \ref{prop234}, we have that $M_{\alpha}$ is
an irreducible module $L(0)$ or $L(\frac{1}{2}\Lambda_1+\frac{1}{2}\Lambda_2)$, that is,
$M$ is $L(-\frac{1}{2},0)$ or $L(-\frac{1}{2},\frac{1}{2}\Lambda_1+\frac{1}{2}\Lambda_2)$.
Similar to the arguments in Theorem \ref{thm1}, we get that the category of ordinary $L_{\widehat{\mathfrak{g}}}(-\frac{1}{2},0)$-modules is semisimple.
\end{proof}

\subsection{Category $\mathcal{C}_{-\frac{1}{2}}$}

Denote by $\mathcal{C}_{\mathcal{k}}$ the full subcategory of weak $L_{\widehat{\mathfrak{g}}}(\mathcal{k},0)$-module category
such that $M$ is an object of $\mathcal{C}_{\mathcal{k}}$ if and only if $a$ acts locally nilpotently on $M$ for all $a\in N_+$.
It is clear that the object of $\mathcal{O}_{\mathcal{k}}$ is also an object of $\mathcal{C}_{\mathcal{k}}$.

For a $\mathfrak{h}$-module $U$, we define a linear function on $U^*\otimes M(\mathcal{k},U)$ as follows:
$$\langle u^\prime, u\rangle=u^\prime(\pi^\prime(u)), ~\mbox{for}~~u^\prime\in U^*,u\in M(\mathcal{k},U),$$
where $\pi^\prime$ is the projection of $M(\mathcal{k},U)$ onto the subspace $U$.
Define
\begin{equation*}
I=\{u\in M(\mathcal{k},U)\mid\langle u^\prime, au\rangle=0~~\mbox{for~~any}~~u^\prime\in U^*,a\in U(\tilde{\mathfrak{g}})\}.
\end{equation*}
It is clear that
$I$ is unique and maximal in the set of submodules of $M(\mathcal{k},U)$ which intersects with $U$ trivially (cf. \cite{LW}).
Set $L(\mathcal{k},U)=M(\mathcal{k},U)/I$ and we regard $U$ as a subspace of $L(\mathcal{k},U)$.
Then $\pi^\prime$ induces a projection of $L(\mathcal{k},U)$ to $U$, which we still denote it by $\pi^\prime$.
It is clear that $L(\mathcal{k},U)$ is a weak module for $V_{\widehat{\mathfrak{g}}}(\mathcal{k},0)$.
Let $Y(\cdot,z)$ be the vertex operator map defining the module structure on $L(\mathcal{k},U)$,
then $Y(\cdot,z)$ is an intertwining operator of type $\begin{pmatrix}L(\mathcal{k},U)\\ V_{\widehat{\mathfrak{g}}}(\mathcal{k},0)L(\mathcal{k},U)\end{pmatrix}$.
Let $\mathcal{Y}(u,z)v=(-1)^{|u||v|}e^{zL(-1)}Y(v,-z)u$
for any homogeneous element $u\in L(\mathcal{k},U),v\in V_{\widehat{\mathfrak{g}}}(\mathcal{k},0)$,
then $\mathcal{Y}(\cdot,z)$ is an intertwining operator of type
$\begin{pmatrix}L(\mathcal{k},U)\\ L(\mathcal{k},U)V_{\widehat{\mathfrak{g}}}(\mathcal{k},0)\end{pmatrix}$ (cf. \cite{FHL}).
It is clear that for any $u^\prime\in U^*, u\in L(\mathcal{k},U), a\in N_-U(N_-), w\in V_{\widehat{\mathfrak{g}}}(\mathcal{k},0)$,
we have $\pi^\prime(a\mathcal{Y}(u,z)w)=0$, then
$\langle u^\prime,a\mathcal{Y}(u,z)w\rangle=0$.

\begin{prop}\label{propmodule}
$L(-\frac{1}{2},U)$ is a weak $L_{\widehat{\mathfrak{g}}}(-\frac{1}{2},0)$-module if and only if $p_1(h_1,h_2)U=0$ and $p_2(h_1,h_2)U=0$, where
\begin{equation*}
p_1(h_1,h_2)=h_1(2h_1-4h_2+1),~p_2(h_1,h_2)=h_2(2h_2-4h_1+1).
\end{equation*}
\end{prop}
\begin{proof}
Similar to the proof of Lemma 2.9 of \cite{DLM2} (see also the Lemma 3.2 of \cite{LW}), we have the $\tilde{\mathfrak{g}}$-module $L(-\frac{1}{2},U)$ is a weak module for $L_{\widehat{\mathfrak{g}}}(-\frac{1}{2},0)$
if and only if
$\langle u^\prime,\mathcal{Y}(u,z)v_{1}\rangle=0$ for any $u^\prime\in U^*, u\in U(\mathfrak{g})U\subset L(-\frac{1}{2},U)$.

Let $\sigma$ be the anti-automorphism of superalgebra $U(\mathfrak{g})$ with $\sigma(ab)=(-1)^{|a||b|}\sigma(b)\sigma(a)$ for any $a,b\in U(\mathfrak{g})$
such that $\sigma(a)=-a$ for any $a\in\mathfrak{g}$,
$\pi$ the projection $\tilde{\mathfrak{g}}$ onto $\mathfrak{g}$ such that $\pi(a\otimes t^n)=a$ for any $a\in\mathfrak{g}$ and $\pi(c)=0$.
Define $\mbox{deg}(a\otimes t^n)=n$ for any $n\in\mathbb{Z}, a\in\mathfrak{g}$.
For any $u^\prime\in U^*, a(m)\in N_-, u\in U(\mathfrak{g})U\subset L(-\frac{1}{2},U), w\in V_{\widehat{\mathfrak{g}}}(-\frac{1}{2},0)$, we have
\begin{align*}
 \langle u^\prime,\mathcal{Y}(u,z)a(m)w\rangle&=(-1)^{|u||a|}(\langle u^\prime,a(m)\mathcal{Y}(u,z)w\rangle-\langle u^\prime,[a(m),\mathcal{Y}(u,z)]w\rangle)\\
 &=(-1)^{|u||a|+1}\langle u^\prime,\sum_{j\geq0}\begin{pmatrix}m\\j\end{pmatrix}\mathcal{Y}(a(j)u,z)z^{m-j}w\rangle\\
 &=(-1)^{|u||a|+1}\langle u^\prime,\mathcal{Y}(a(0)u,z)z^{m}w\rangle\\
 &=\langle u^\prime,(-1)^{|u||a|+1}z^{\mbox{deg}(a(m))}\mathcal{Y}(\pi(a(m))u,z)w\rangle.
\end{align*}
Hence for any $u^\prime\in U^*, a\in U(N_-), u\in U(\mathfrak{g})U\subset L(-\frac{1}{2},U), w\in V_{\widehat{\mathfrak{g}}}(-\frac{1}{2},0)$ we have
$$ \langle u^\prime,\mathcal{Y}(u,z)aw\rangle=\langle u^\prime,\xi z^{\mbox{deg}(a)}\mathcal{Y}(\sigma\pi(a)u,z)w\rangle,$$
where $\xi=1$ or $-1$.
Let $a=2e_1(-1)e_2(-1)+2h_-(-1)e_{12}(-1)-e_{12}(-2)$, then $v_1=a\textbf{1}$.
Hence $L(-\frac{1}{2},U)$ is a weak $L_{\widehat{\mathfrak{g}}}(-\frac{1}{2},0)$-module if and only if
$\langle u^\prime,\mathcal{Y}(\sigma\pi(a)u,z)\textbf{1}\rangle=0$ for any $u^\prime\in U^*, u\in U(\mathfrak{g})U\subset L(-\frac{1}{2},U)$.
Note that $$\mathcal{Y}(\sigma\pi(a)u,z)\textbf{1}=e^{zL(-1)}Y(\textbf{1},-z)\sigma\pi(a)u=e^{zL(-1)}\sigma\pi(a)u,$$
thus for any $u^\prime\in U^*, u\in U(\mathfrak{g})U\subset L(-\frac{1}{2},U)$, we have
\begin{equation*}
\langle u^\prime,\mathcal{Y}(\sigma\pi(a)u,z)\textbf{1}\rangle=\langle u^\prime,e^{zL(-1)}\sigma\pi(a)u\rangle=\langle u^\prime,\sigma\pi(a)u\rangle.
\end{equation*}
Since $\sigma\pi(a)=-2e_2e_1+2e_{12}h_-+e_{12}$,
we have $L(-\frac{1}{2},U)$ is a weak $L_{\widehat{\mathfrak{g}}}(-\frac{1}{2},0)$-module if and only if for any $u^\prime\in U^*$,
$\langle u^\prime, -2e_2e_1+2e_{12}h_-+e_{12}U(\mathfrak{g})U\rangle=0$.
From the grading restriction on the bilinear pair, it is equivalent to
$(-2e_2e_1+2e_{12}h_-+e_{12})f_1f_2U=0$ and $(-2e_2e_1+2e_{12}h_-+e_{12})f_2f_1U=0$.
Since $\mathfrak{n}_+U=0$, we have
$L(-\frac{1}{2},U)$ is a weak $L_{\widehat{\mathfrak{g}}}(-\frac{1}{2},0)$-module if and only if
$p_1(h_1,h_2)U=0$ and $p_2(h_1,h_2)U=0$, where $p_1(h_1,h_2)=h_1(2h_1-4h_2+1),~p_2(h_1,h_2)=h_2(2h_2-4h_1+1)$.
\end{proof}

Let $\Omega(M)=\{v\in M\mid (\mathfrak{g}\otimes t\mathbb{C}[t]).v=0\}$
for any weak $L_{\widehat{\mathfrak{g}}}(-\frac{1}{2},0)$-module $M$.

\begin{prop}\label{prop11}
Let $M$ be a weak $L_{\widehat{\mathfrak{g}}}(-\frac{1}{2},0)$-module belonging to $\mathcal{C}_{-\frac{1}{2}}$, then $M$ has a highest weight vector.
\end{prop}
\begin{proof}
Similar to the proof of Proposition 3.6 of \cite{Lin} or that of the Theorem 3.7 of \cite{DLM1}, we have $\Omega(M)\ne0$.
Since $M$ is an object of $\mathcal{C}_{-\frac{1}{2}}$,
then $U(\mathfrak{n}_+)v$ is finite dimensional for any $v\in\Omega(M)$.
Since $\mathfrak{n}_+$ is a nilpotent Lie superalgebra,
by the Engel theorem for Lie superalgebra (cf. \cite{Kac}), there exists a nonzero $w\in U(\mathfrak{n}_+)v$ such that $\mathfrak{n}_+.w=0$.
Since $\mathfrak{n}_+.\Omega(M)\subseteq\Omega(M)$, we have $w\in\Omega(M)$, it implies $N_+.w=0$.
Furthermore,
let $U=U(\mathfrak{h})w$ and $W$ be the weak $L_{\widehat{\mathfrak{g}}}(-\frac{1}{2},0)$-submodule generated by $U$.
Thus $W$ is isomorphic to a quotient module of $M(-\frac{1}{2},U)$ as $\tilde{\mathfrak{g}}$-module,
it follows that $L(-\frac{1}{2},U)$ is a quotient of $W$ as weak $V_{\widehat{\mathfrak{g}}}(-\frac{1}{2},0)$-module.
Therefore $L(-\frac{1}{2},U)$ is a weak $L_{\widehat{\mathfrak{g}}}(-\frac{1}{2},0)$-module.
From the Proposition \ref{propmodule}, we have $p_1(h_1,h_2)U=0$ and $p_2(h_1,h_2)U=0$.
Thus $U$ is a $\mathbb{C}[h_1,h_2]\Big/\langle p_1(h_1,h_2),p_2(h_1,h_2)\rangle$-module.
Let $I_1$ be the ideal of $\mathbb{C}[h_1,h_2]$ generated by $h_1,h_2$, i.e.,
$I_1=\langle h_1,h_2\rangle$, let $I_2=\langle h_1,h_2+\frac{1}{2}\rangle,I_3=\langle h_1+\frac{1}{2},h_2\rangle,I_4=\langle h_1-\frac{1}{2},h_2-\frac{1}{2}\rangle$,
it is clear that $I_i,I_j$ are coprime whenever $i\ne j$, thus $\langle p_1(h_1,h_2),p_2(h_1,h_2)\rangle=\prod_{i=1}^{4}I_i=\bigcap_{i=1}^{4}I_i$.
Hence
\begin{equation*}
\mathbb{C}[h_1,h_2]\Big/\langle p_1(h_1,h_2),p_2(h_1,h_2)\rangle\cong\prod_{i=1}^{4}\mathbb{C}[h_1,h_2]\big/I_i,
\end{equation*}
then $U=\bigoplus_{i=1}^{4}U_i$, where $U_i$ is a $\mathbb{C}[h_1,h_2]\big/I_i$-module, i.e., $I_iU_i=0$.
Therefore $\mathfrak{h}$ acts semisimply on $U$, i.e., $w$ is a highest weight vector.
\end{proof}

Now we show that the category $\mathcal{C}_{-\frac{1}{2}}$ is semisimple and has finitely many irreducible modules.

\begin{thm}\label{thmc12}
The set $\{L(-\frac{1}{2},\lambda)|\lambda=0,-\frac{1}{2}\Lambda_1,-\frac{1}{2}\Lambda_2,\frac{1}{2}\Lambda_1+\frac{1}{2}\Lambda_2\}$
provides the complete list of irreducible modules in the category $\mathcal{C}_{-\frac{1}{2}}$.
Moreover, the category $\mathcal{C}_{-\frac{1}{2}}$ is semisimple.
\end{thm}
\begin{proof}
Let $M$ be an irreducible weak $L_{\widehat{\mathfrak{g}}}(-\frac{1}{2},0)$-module belonging to $\mathcal{C}_{-\frac{1}{2}}$,
then by the Proposition \ref{prop11}, $M$ contains a highest weight vector $w$.
Since $M$ is an irreducible weak $L_{\widehat{\mathfrak{g}}}(-\frac{1}{2},0)$-module,
we have that $M$ is an irreducible weak $V_{\widehat{\mathfrak{g}}}(-\frac{1}{2},0)$-module,
thus $M$ is an irreducible $\tilde{\mathfrak{g}}$-module generated by $w$.
It shows that $M$ is an irreducible highest weight module of $\tilde{\mathfrak{g}}$.
From the Proposition \ref{propmodule}, $M$ is the module $L(-\frac{1}{2},0)$, $L(-\frac{1}{2},-\frac{1}{2}\Lambda_1)$, $L(-\frac{1}{2},-\frac{1}{2}\Lambda_2)$ or $L(-\frac{1}{2},\frac{1}{2}\Lambda_1+\frac{1}{2}\Lambda_2)$.

Now we prove that the category $\mathcal{C}_{-\frac{1}{2}}$ is semisimple.
Let $M$ be a weak $L_{\widehat{\mathfrak{g}}}(-\frac{1}{2},0)$-module belonging to $\mathcal{C}_{-\frac{1}{2}}$.
From the Proposition \ref{prop11}, $M$ contains a highest weight vector $w$.
Let $W$ be the weak $L_{\widehat{\mathfrak{g}}}(-\frac{1}{2},0)$-submodule generated by $w$,
then $W$ is a quotient of certain Verma module $M(-\frac{1}{2},\lambda^\prime)$ as $\tilde{\mathfrak{g}}$-module.
Hence $W$ is an object in $\mathcal{O}_{-\frac{1}{2}}$.
From the Theorem \ref{thm1}, it is completely reducible, i.e., $W=\bigoplus_\lambda L(-\frac{1}{2},\lambda)$.
But the highest weight subspace of $M(-\frac{1}{2},\lambda^\prime)$ is one dimensional,
then $W=L(-\frac{1}{2},\lambda^\prime)$.
Let $W^\prime$ be the sum of irreducible weak $L_{\widehat{\mathfrak{g}}}(-\frac{1}{2},0)$-submodule of $M$.
If $W^\prime$ is a proper submodule of $M$,
then $M/W^\prime$ is a weak $L_{\widehat{\mathfrak{g}}}(-\frac{1}{2},0)$-module belonging to $\mathcal{C}_{-\frac{1}{2}}$,
thus it contains a highest weight vector $\bar{w^\prime}$ by the Proposition \ref{prop11}.
Let $w^\prime$ be a preimage of $\bar{w^\prime}$, thus $N_+.w^\prime\subseteq W^\prime$.
Since $N_+$ is finitely generated,
there exist submodules $L(-\frac{1}{2},\lambda_1), L(-\frac{1}{2},\lambda_2),\cdots, L(-\frac{1}{2},\lambda_s)$ of $W^\prime$
such that $$N_+.w^\prime\subseteq L(-\frac{1}{2},\lambda_1)\oplus L(-\frac{1}{2},\lambda_2)\oplus\cdots\oplus L(-\frac{1}{2},\lambda_s).$$
From the Theorem \ref{thm1},
the submodule of $M$ generated by $w^\prime$ and
$L(-\frac{1}{2},\lambda_1)\oplus L(-\frac{1}{2},\lambda_2)\oplus\cdots\oplus L(-\frac{1}{2},\lambda_s)$ is completely reducible.
Thus the submodule of $M$ generated by $w^\prime$ is a direct sum of certain $L(-\frac{1}{2},\lambda)$,
it is a contradiction. Therefore $W^\prime=M$, i.e., $M$ is completely reducible,
i.e., the category $\mathcal{C}_{-\frac{1}{2}}$ is semisimple.
\end{proof}
\section{$\mathbb{Q}$-graded vertex operator superalgebras}
\label{sec:4}
	\def\theequation{4.\arabic{equation}}
	\setcounter{equation}{0}

In this section,
let $0<\xi<1$ be a rational number,
we prove that $(L_{\widehat{\mathfrak{g}}}(-\frac{1}{2},0),\omega_\xi)$ is a rational $\mathbb{Q}$-graded vertex operator superalgebra.
Then we determine the Zhu's algebra
$A_{\omega_\xi}(L_{\widehat{\mathfrak{g}}}(-\frac{1}{2},0))$
of $(L_{\widehat{\mathfrak{g}}}(-\frac{1}{2},0),\omega_\xi)$ and prove that
$A_{\omega_\xi}(L_{\widehat{\mathfrak{g}}}(-\frac{1}{2},0))$ is a finite-dimensional semisimple associate superalgebra.
Finally, we show that $(L_{\widehat{\mathfrak{g}}}(-\frac{1}{2},0),\omega_\xi)$ is $C_2$-cofinite.

\begin{thm}\label{thm22}
Let $0<\xi<1$ be a rational number, then $(L_{\widehat{\mathfrak{g}}}(-\frac{1}{2},0),\omega_\xi)$ is a rational $\mathbb{Q}$-graded vertex operator superalgebra.
\end{thm}
\begin{proof}
Let $M$ be a $\mathbb{Q}_+$-graded weak $(L_{\widehat{\mathfrak{g}}}(-\frac{1}{2},0),\omega_\xi)$-module,
we denote by $\mbox{deg}~w$ the degree of homogeneous element $w\in M$.
For any $\alpha\in \Delta_{\mathfrak{g}}^+$ and the corresponding root vector $x_\alpha$,
 $$L^\prime(0)x_\alpha=L(0)x_\alpha-\frac{\xi}{2}h_0x_\alpha=(1-\frac{\xi}{2}\alpha(h))x_\alpha.$$
 Since $0<\xi<1$, we have
$\mbox{deg}~x_\alpha w=\mbox{wt}~x_\alpha+\mbox{deg}~w-1<\mbox{deg}~w$ for all $w\in M$,
it implies that $a$ act locally nilpotently on $M$ for all $a\in\mathfrak{n}_+$.
For any $a(n) \in\mathfrak{g}\otimes t\mathbb{C}[t]$, we have
$\mbox{deg}~a(n)w=\mbox{deg}~w+\mbox{wt}~a-n-1<\mbox{deg}~w$
for all $w\in M$, thus $a(n)$ acts locally nilpotently on $M$.
Therefore $a$ acts locally nilpotently on $M$ for all $a\in N_+$, i.e., $M$ is an object of $\mathcal{C}_{-\frac{1}{2}}$.
Thus $M$ is a direct sum of certain simple objects of $\mathcal{C}_{-\frac{1}{2}}$ by the Theorem \ref{thmc12}.
Similar to the proof that $(L_{\widehat{\mathfrak{g}}}(-\frac{1}{2},0),\omega_\xi)$ is a $\mathbb{Q}$-graded vertex operator superalgebra,
the simple module $L(-\frac{1}{2},\lambda)$ of $\mathcal{C}_{-\frac{1}{2}}$ is an ordinary $(L_{\widehat{\mathfrak{g}}}(-\frac{1}{2},0),\omega_\xi)$-module,
hence $L(-\frac{1}{2},\lambda)$ is a $\mathbb{Q}_+$-graded weak
$(L_{\widehat{\mathfrak{g}}}(-\frac{1}{2},0),\omega_\xi)$-module.
Therefore, $(L_{\widehat{\mathfrak{g}}}(-\frac{1}{2},0),\omega_\xi)$ is a rational $\mathbb{Q}$-graded vertex operator superalgebra.
\end{proof}

Denote by $A_{\omega_\xi}(V_{\widehat{\mathfrak{g}}}(\mathcal{k},0))$ (resp. $A_{\omega_\xi}(L_{\widehat{\mathfrak{g}}}(\mathcal{k},0))$)
the Zhu's algebra of $\mathbb{Q}$-graded vertex operator superalgebra $(V_{\widehat{\mathfrak{g}}}(\mathcal{k},0),\omega_\xi)$
(resp. ($L_{\widehat{\mathfrak{g}}}(\mathcal{k},0),\omega_\xi)$).

\begin{prop}\label{propav}
The associative superalgebra $A_{\omega_\xi}(V_{\widehat{\mathfrak{g}}}(\mathcal{k},0))$ is isomorphic to the polynomial algebra $\mathbb{C}[t_1,t_2]$.
\end{prop}

\begin{proof}
Since $\mbox{wt}~h_1(-1)\textbf{1}=1,~\mbox{wt}~h_2(-1)\textbf{1}=1,~\mbox{wt}~e_{12}(-1)\textbf{1}=1-\xi,~\mbox{wt}~f_{12}(-1)\textbf{1}=1+\xi, ~\mbox{wt}~e_1(-1)\textbf{1}=1-\frac{1}{2}\xi,~\mbox{wt}~f_1(-1)\textbf{1}=1+\frac{1}{2}\xi,
  ~\mbox{wt}~e_2(-1)\textbf{1}=1-\frac{1}{2}\xi,~\mbox{wt}~f_2(-1)\textbf{1}=1+\frac{1}{2}\xi$,
we have
\begin{align}\label{fyexh}
&\mbox{Res}_z\frac{(1+z)^{[{\rm wt}~f_{12}]}}{z^{m}}Y(f_{12},z)u=(f_{12}(-m)+f_{12}(1-m))u,\notag\\
&\mbox{Res}_z\frac{(1+z)^{[{\rm wt}~f_1]}}{z^{m}}Y(f_1,z)u=(f_1(-m)+f_1(1-m))u,\notag\\
&\mbox{Res}_z\frac{(1+z)^{[{\rm wt}~f_2]}}{z^{m}}Y(f_2,z)u=(f_2(-m)+f_2(1-m))u,\notag\\
&\mbox{Res}_z\frac{(1+z)^{[{\rm wt}~e_{12}]}}{z^{m}}Y(e_{12},z)u=e_{12}(-m)u,\notag\\
&\mbox{Res}_z\frac{(1+z)^{[{\rm wt}~e_1]}}{z^{m}}Y(e_1,z)u=e_1(-m)u,\\
&\mbox{Res}_z\frac{(1+z)^{[{\rm wt}~e_2]}}{z^{m}}Y(e_2,z)u=e_2(-m)u,\notag\\
&\mbox{Res}_z\frac{(1+z)^{[{\rm wt}~h_1]}}{z^{m+1}}Y(h_1,z)u=(h_1(-m-1)+h_1(-m))u,\notag\\
&\mbox{Res}_z\frac{(1+z)^{[{\rm wt}~h_2]}}{z^{m+1}}Y(h_2,z)u=(h_2(-m-1)+h_2(-m))u\notag
\end{align}
for any positive integer $m$ and $u\in V_{\widehat{\mathfrak{g}}}(\mathcal{k},0)$.
It is clear that
all those elements in (\ref{fyexh}) are in $O_{\omega_\xi}(V_{\widehat{\mathfrak{g}}}(\mathcal{k},0))$.
Let $W$ be the subspace linearly spanned by the elements in (\ref{fyexh}),
then $W=CV_{\widehat{\mathfrak{g}}}(\mathcal{k},0)$, where $C=\mathfrak{n}_-\otimes\mathbb{C}[t^{-1}](t^{-1}+1)+\mathfrak{n}_+\otimes\mathbb{C}[t^{-1}]t^{-1}
+\mathfrak{h}\otimes\mathbb{C}[t^{-1}](t^{-2}+t^{-1})$.
Let $L$ be the linear span of homogeneous elements $a$ of $(V_{\widehat{\mathfrak{g}}}(\mathcal{k},0),\omega_\xi)$
such that for any positive integer $n$,
\begin{equation*}
\mbox{Res}_z\frac{(1+z)^{[{\rm wt}~a]}}{z^{n+\varepsilon(a)}}Y(a,z)V_{\widehat{\mathfrak{g}}}(\mathcal{k},0)\subseteq W.
\end{equation*}
Let $a$ be any homogeneous element of $L$,
similar to the proof of the Proposition 4.1 of \cite{DLM2},
we have $h_1(-m)a,h_2(-m)a,e_{12}(-m)a,f_{12}(-m)a,e_1(-m)a,f_1(-m)a,e_2(-m)a,f_2(-m)a$ $\in L$ for any positive integer $m$.
Since $\textbf{1}\in L$ and $V_{\widehat{\mathfrak{g}}}(\mathcal{k},0)=U(\mathfrak{g}\otimes t^{-1}\mathbb{C}[t^{-1}])\textbf{1}$,
we have $L=V_{\widehat{\mathfrak{g}}}(\mathcal{k},0)$.
Therefore $O_{\omega_\xi}(V_{\widehat{\mathfrak{g}}}(\mathcal{k},0))=W=CV_{\widehat{\mathfrak{g}}}(\mathcal{k},0)$.

Define a linear map
\begin{equation*}
 \begin{aligned}
 \psi:~&\mathbb{C}[t_1,t_2]\rightarrow A_{\omega_\xi}(V_{\widehat{\mathfrak{g}}}(\mathcal{k},0))\\
 &g(t_1,t_2)\mapsto g(h_1(-1),h_2(-1))\textbf{1}+O_{\omega_\xi}(V_{\widehat{\mathfrak{g}}}(\mathcal{k},0))
 \end{aligned}
\end{equation*}
for any $g(t_1,t_2)\in \mathbb{C}[t_1,t_2]$.
Since $h_1(0)\textbf{1}=h_2(0)\textbf{1}=0$,
we have $g(h_1(-1),h_2(-1))\textbf{1}=g(h_1(-1)+h_1(0),h_2(-1)+h_2(0))\textbf{1}$ for any $g(t_1,t_2)\in \mathbb{C}[t_1,t_2]$.
Since $h_i(-1)\textbf{1}\ast u=(h_i(-1)+h_i(0))u$ for any $i=1,2,u\in V_{\widehat{\mathfrak{g}}}(\mathcal{k},0)$, we have
\begin{align*}
  &\psi(f(t_1,t_2)g(t_1,t_2))\\=&f(h_1(-1),h_2(-1))g(h_1(-1),h_2(-1))\textbf{1}+O_{\omega_\xi}(V_{\widehat{\mathfrak{g}}}(\mathcal{k},0))\\
  =&f(h_1(-1)+h_1(0),h_2(-1)+h_2(0))g(h_1(-1),h_2(-1))\textbf{1}+O_{\omega_\xi}(V_{\widehat{\mathfrak{g}}}(\mathcal{k},0))\\
  =&f(h_1(-1),h_2(-1))\textbf{1}\ast g(h_1(-1),h_2(-1))\textbf{1}+O_{\omega_\xi}(V_{\widehat{\mathfrak{g}}}(\mathcal{k},0))\\=&\psi(f(t_1,t_2))\ast\psi(g(t_1,t_2)),
\end{align*}
thus $\psi$ is a superalgebra homomorphism.
Note that $N_-=C\oplus\mathfrak{n}_-\oplus\mathbb{C}h_1(-1)\oplus\mathbb{C}h_2(-1)$,
we have $U(N_-)=U(C)U(\mathbb{C}h_1(-1))U(\mathbb{C}h_2(-1))U(\mathfrak{n}_-)$.
Thus
\begin{equation*}
  O_{\omega_\xi}(V_{\widehat{\mathfrak{g}}}(\mathcal{k},0))=CV_{\widehat{\mathfrak{g}}}(\mathcal{k},0)=CU(N_-)\textbf{1}\cong CU(C)U(\mathbb{C}h_1(-1))U(\mathbb{C}h_2(-1)).
\end{equation*}
Therefore $\psi$ is an isomorphism.
\end{proof}

Now we determine the Zhu's algebra $A_{\omega_\xi}(L_{\widehat{\mathfrak{g}}}(-\frac{1}{2},0))$.

\begin{thm}\label{thmav}
The associative superalgebra $A_{\omega_\xi}(L_{\widehat{\mathfrak{g}}}(-\frac{1}{2},0))$
is isomorphic to the quotient algebra $\mathbb{C}[t_1,t_2]/\langle p_1(t_1,t_2),p_2(t_1,t_2)\rangle$,
where
\begin{equation*}
p_1(t_1,t_2)=t_1(2t_1-4t_2+1),~p_2(t_1,t_2)=t_2(2t_2-4t_1+1).
\end{equation*}
Moreover, $A_{\omega_\xi}(L_{\widehat{\mathfrak{g}}}(-\frac{1}{2},0))$ is semisimple and finite-dimensional.
\end{thm}
\begin{proof}
Since $(L_{\widehat{\mathfrak{g}}}(-\frac{1}{2},0),\omega_\xi)$ is the quotient
vertex operator superalgebra of $(V_{\widehat{\mathfrak{g}}}(-\frac{1}{2},0),\omega_\xi)$,
thus by the Lemma \ref{lemfz}, we have that $A_{\omega_\xi}(L_{\widehat{\mathfrak{g}}}(-\frac{1}{2},0))$ is
the quotient superalgebra of $A_{\omega_\xi}(V_{\widehat{\mathfrak{g}}}(-\frac{1}{2},0))$.
Set $A_{\omega_\xi}(L_{\widehat{\mathfrak{g}}}(-\frac{1}{2},0))=A_{\omega_\xi}(V_{\widehat{\mathfrak{g}}}(-\frac{1}{2},0))/I$.
From the proof of Proposition \ref{propav}, any irreducible $A_{\omega_\xi}(V_{\widehat{\mathfrak{g}}}(-\frac{1}{2},0))$-module
$U$ is also an irreducible $\mathfrak{h}$-module.
By the definition of $L(-\frac{1}{2}, U)$,
$L(-\frac{1}{2}, U)$ is the corresponding irreducible $\mathbb{Q}_+$-graded weak $V_{\widehat{\mathfrak{g}}}(-\frac{1}{2},0)$-module
in the Theorem \ref{thmzhu}.
From the Proposition \ref{propmodule},
$L(-\frac{1}{2},U)$ is an irreducible $\mathbb{Q}_+$-graded weak $L_{\widehat{\mathfrak{g}}}(-\frac{1}{2},0)$-module
if and only if $p_1(h_1,h_2)U=p_2(h_1,h_2)U=0$.
On the other hand, $L(-\frac{1}{2},U)$ is an irreducible $\mathbb{Q}_+$-graded weak $L_{\widehat{\mathfrak{g}}}(-\frac{1}{2},0)$-module
is equivalent to $U$ is an irreducible $A_{\omega_\xi}(L_{\widehat{\mathfrak{g}}}(-\frac{1}{2},0))$-module, i.e., $I.U=0$.
Thus $\psi^{-1}(I)=\langle p_1(t_1,t_2),p_2(t_1,t_2)\rangle$.
Therefore, $A_{\omega_\xi}(L_{\widehat{\mathfrak{g}}}(-\frac{1}{2},0))$ is isomorphic to $$\mathbb{C}[t_1,t_2]/\langle p_1(t_1,t_2),p_2(t_1,t_2)\rangle.$$
From the proof of the Proposition \ref{prop11},
$A_{\omega_\xi}(L_{\widehat{\mathfrak{g}}}(-\frac{1}{2},0))$ is semisimple and finite-dimensional.
\end{proof}

Finally, we show that $(L_{\widehat{\mathfrak{g}}}(-\frac{1}{2},0),\omega_\xi)$ is $C_2$-cofinite.

\begin{thm}
The commutative associative superalgebra $A_2(L_{\widehat{\mathfrak{g}}}(-\frac{1}{2},0))$ for $(L_{\widehat{\mathfrak{g}}}(-\frac{1}{2},0),\omega_\xi)$
is isomorphic to the quotient algebra $\mathbb{C}[t_1,t_2]/\langle q_1(t_1,t_2),q_2(t_1,t_2)\rangle$,
where
\begin{equation*}
q_1(t_1,t_2)=t_1(t_1-2t_2),~q_2(t_1,t_2)=t_2(t_2-2t_1).
\end{equation*}
Moreover, $(L_{\widehat{\mathfrak{g}}}(-\frac{1}{2},0),\omega_\xi)$ is $C_2$-cofinite.
\end{thm}
\begin{proof}
Let $B_1=\mathbb{C}e_{12}(-1)+\mathbb{C}e_{1}(-1)+\mathbb{C}e_{2}(-1)+\mathbb{C}f_{12}(-1)+\mathbb{C}f_{1}(-1)+\mathbb{C}f_{2}(-1)+
\mathfrak{g}\otimes t^{-2}\mathbb{C}[t^{-1}]$.
Similar to the Proposition 4.9 of \cite{DLM2},
we have $C_2(M)=B_1M$ for any weak $(V_{\widehat{\mathfrak{g}}}(-\frac{1}{2},0),\omega_\xi)$-module $M$.
If we can also view $M$ as a $(L_{\widehat{\mathfrak{g}}}(-\frac{1}{2},0),\omega_\xi)$-module, by the definition of $C_2(M)$,
we have $C_2(M)=B_1M$ for $(L_{\widehat{\mathfrak{g}}}(-\frac{1}{2},0),\omega_\xi)$-module $M$.
Let $B_2=\mathbb{C}f_{12}(-1)+\mathbb{C}f_{1}(-1)+\mathbb{C}f_{2}(-1)+
\mathfrak{g}\otimes t^{-2}\mathbb{C}[t^{-1}]$,
it is clear that $B_2$ is an ideal of $N_-$.
Set $L=N_-/B_2$ and
denote by $\bar{e}_{12}=e_{12}(-1)+B_2,\bar{e}_{1}=e_{1}(-1)+B_2,\bar{e}_{2}=e_{2}(-1)+B_2,\bar{f}_{12}=f_{12}(0)+B_2,\bar{f}_{1}=f_{1}(0)+B_2,
\bar{f}_{2}=f_{2}(0)+B_2,\bar{h}_{1}=h_{1}(-1)+B_2,\bar{h}_{2}=h_{2}(-1)+B_2$.
Let $\bar{\mathfrak{h}}=\mathbb{C}\bar{h}_{1}\oplus\mathbb{C}\bar{h}_{2}$
and $\bar{h}_{-}=\bar{h}_{1}-\bar{h}_{2}, \bar{h}_{+}=\bar{h}_{1}+\bar{h}_{2}$.

Since $M(-\frac{1}{2},0)\cong U(N_-)$ as vector space,
then
\begin{align*}
 C_2(M(-\frac{1}{2},0))&=B_1M(-\frac{1}{2},0)\\&=B_2M(-\frac{1}{2},0)+e_{12}(-1)M(-\frac{1}{2},0)+e_{1}(-1)M(-\frac{1}{2},0)+e_{2}(-1)M(-\frac{1}{2},0)
\\&\cong B_2U(N_-)+e_{12}(-1)U(N_-)+e_{1}(-1)U(N_-)+e_{2}(-1)U(N_-)
\end{align*}
as vector space.
Since $L(-\frac{1}{2},0)=M(-\frac{1}{2},0)\big/(U(N_-)av+U(N_-)f_1(0)v+U(N_-)f_2(0)v)$,
where $a=2e_1(-1)e_2(-1)+2h_-(-1)e_{12}(-1)-e_{12}(-2)$ and $v$ is a highest weight vector of $M(-\frac{1}{2},0)$,
we have
\begin{equation*}
\begin{aligned}
C_2(L(-\frac{1}{2},0))=B_1L(-\frac{1}{2},0)\cong &(B_1M(-\frac{1}{2},0)+U(N_-)av+U(N_-)f_1(0)v\\&+U(N_-)f_2(0)v)\Big/(U(N_-)av+U(N_-)f_1(0)v+U(N_-)f_2(0)v).
\end{aligned}
\end{equation*}
Hence
\begin{align*}
 &A_2(L(-\frac{1}{2},0))=L(-\frac{1}{2},0)\big/C_2(L(-\frac{1}{2},0))\\ &\cong
M(-\frac{1}{2},0)\Big/(B_1M(-\frac{1}{2},0)+U(N_-)av+U(N_-)f_1(0)v+U(N_-)f_2(0)v)
\\&\cong U(N_-)\Big/\begin{pmatrix}B_2U(N_-)+e_{12}(-1)U(N_-)+e_{1}(-1)U(N_-)+e_{2}(-1)U(N_-)\\+U(N_-)a+U(N_-)f_1(0)+U(N_-)f_2(0)\end{pmatrix}
\\&\cong U(L)\Big/(\bar{e}_{12}U(L)+\bar{e}_{1}U(L)+\bar{e}_{2}U(L)+
U(L)(\bar{e}_{1}\bar{e}_{2}+\bar{h}_{-}\bar{e}_{12})+U(L)\bar{f}_{1}+U(L)\bar{f}_{2})
\end{align*}
as a vector space.
We consider the element
$$\bar{e}_{1}^{a_1}\bar{e}_{2}^{a_2}\bar{e}_{12}^{a_3}\bar{h}_{1}^{b_1}
\bar{h}_{2}^{b_2}\bar{f}_{1}^{c_1}\bar{f}_{2}^{c_2}\bar{f}_{12}^{c_3}(\bar{e}_{1}\bar{e}_{2}+\bar{h}_{-}\bar{e}_{12}),$$
where $a_1,a_2,c_1,c_2\in\{0,1\},b_1,b_2,a_3,c_3\in\mathbb{Z}_+$.
If $c_3\geq 1$,
\begin{align*}
\bar{e}_{1}^{a_1}\bar{e}_{2}^{a_2}\bar{e}_{12}^{a_3}\bar{h}_{1}^{b_1}
\bar{h}_{2}^{b_2}\bar{f}_{1}^{c_1}\bar{f}_{2}^{c_2}\bar{f}_{12}^{c_3}(\bar{e}_{1}\bar{e}_{2}+\bar{h}_{-}\bar{e}_{12})
&\equiv-\bar{e}_{1}^{a_1}\bar{e}_{2}^{a_2}\bar{e}_{12}^{a_3}\bar{h}_{1}^{b_1}
\bar{h}_{2}^{b_2}\bar{f}_{1}^{c_1}\bar{f}_{2}^{c_2}\bar{h}_{-}\bar{h}_{+}\bar{f}_{12}^{c_3-1}\\
&(\mbox{mod}~(\bar{e}_{12}U(L)+\bar{e}_{1}U(L)+\bar{e}_{2}U(L)+U(L)\bar{f}_{1}+U(L)\bar{f}_{2})),
\end{align*}
thus for $a_i>0~(i=1,2,3)$ or $c_3>1$ or $c_1=c_3=1$ or $c_2=c_3=1$, we have
$\bar{e}_{1}^{a_1}\bar{e}_{2}^{a_2}\bar{e}_{12}^{a_3}\bar{h}_{1}^{b_1}
\bar{h}_{2}^{b_2}\bar{f}_{1}^{c_1}\bar{f}_{2}^{c_2}\bar{f}_{12}^{c_3}$ $(\bar{e}_{1}\bar{e}_{2}+\bar{h}_{-}\bar{e}_{12})\in
\bar{e}_{12}U(L)+\bar{e}_{1}U(L)+\bar{e}_{2}U(L)+U(L)\bar{f}_{1}+U(L)\bar{f}_{2}$.
If $a_i=0~(i=1,2,3),c_3=0,c_1=c_2=1$, we have
\begin{align*}
\bar{h}_{1}^{b_1}\bar{h}_{2}^{b_2}\bar{f}_{1}\bar{f}_{2}(\bar{e}_{1}\bar{e}_{2}+\bar{h}_{-}\bar{e}_{12})
\equiv&\bar{h}_{1}^{b_1}\bar{h}_{2}^{b_2}\bar{h}_1(\bar{h}_1-2\bar{h}_2)\\
&(\mbox{mod}~(\bar{e}_{12}U(L)+\bar{e}_{1}U(L)+\bar{e}_{2}U(L)+U(L)\bar{f}_{1}+U(L)\bar{f}_{2})).
\end{align*}
Similarly, we have $\bar{h}_{1}^{b_1}\bar{h}_{2}^{b_2}\bar{f}_{1}(\bar{e}_{1}\bar{e}_{2}+\bar{h}_{-}\bar{e}_{12}),
\bar{h}_{1}^{b_1}\bar{h}_{2}^{b_2}\bar{f}_{2}(\bar{e}_{1}\bar{e}_{2}+\bar{h}_{-}\bar{e}_{12}),
\bar{h}_{1}^{b_1}\bar{h}_{2}^{b_2}(\bar{e}_{1}\bar{e}_{2}+\bar{h}_{-}\bar{e}_{12})\in
\bar{e}_{12}U(L)+\bar{e}_{1}U(L)+\bar{e}_{2}U(L)+U(L)\bar{f}_{1}+U(L)\bar{f}_{2}$.
Thus $\bar{e}_{12}U(L)+\bar{e}_{1}U(L)+\bar{e}_{2}U(L)+
U(L)(\bar{e}_{1}\bar{e}_{2}+\bar{h}_{-}\bar{e}_{12})+U(L)\bar{f}_{1}+U(L)\bar{f}_{2}=
\bar{e}_{12}U(L)+\bar{e}_{1}U(L)+\bar{e}_{2}U(L)+U(L)\bar{f}_{1}+U(L)\bar{f}_{2}+
U(\bar{\mathfrak{h}})\bar{h}_1(\bar{h}_1-2\bar{h}_2)+U(\bar{\mathfrak{h}})(\bar{h}_1-\bar{h}_2)(\bar{h}_1+\bar{h}_2)$.
Therefore
\begin{align*}
 A_2(L(-\frac{1}{2},0))&\cong U(\bar{\mathfrak{h}})\big/(U(\bar{\mathfrak{h}})\bar{h}_1(\bar{h}_1-2\bar{h}_2)+U(\bar{\mathfrak{h}})(\bar{h}_1-\bar{h}_2)(\bar{h}_1+\bar{h}_2))\\
 &\cong U(\bar{\mathfrak{h}})\big/(U(\bar{\mathfrak{h}})\bar{h}_1(\bar{h}_1-2\bar{h}_2)+U(\bar{\mathfrak{h}})\bar{h}_2(\bar{h}_2-2\bar{h}_1))\\
 &\cong\mathbb{C}[t_1,t_2]/\langle q_1(t_1,t_2),q_2(t_1,t_2)\rangle
\end{align*}
as a vector space,
it is clear that this vector space isomorphism is also an associative algebra isomorphism.
From the Finiteness Theorem (cf. \cite{CLO}), we have
$$\mbox{dim}~\mathbb{C}[t_1,t_2]/\langle q_1(t_1,t_2),q_2(t_1,t_2)\rangle<\infty.$$
Hence $(L_{\widehat{\mathfrak{g}}}(-\frac{1}{2},0),\omega_\xi)$ is $C_2$-cofinite.
\end{proof}

\section{Non-boundary admissible level $\frac{1}{2}$}
\label{sec:5}
	\def\theequation{5.\arabic{equation}}
	\setcounter{equation}{0}

In this section, we consider the non--boundary admissible level $\frac{1}{2}$.
We show that the category $\mathcal{O}_{\frac{1}{2}}$ has infinitely many irreducible modules,
the category $\mathcal{C}_{\frac{1}{2}}$ is not semisimple
and the $\mathbb{Q}$-graded vertex operator superalgebra $(L_{\widehat{\mathfrak{g}}}(\frac{1}{2},0),\omega_\xi)$ is not rational.

By using Mathematica,
we obtain the following singular vector of $\widetilde{\mathfrak{g}}$-module $V(\frac{1}{2},\mathbb{C})$.

\begin{align*}
v_2=&(2889/128e_{12}(-2)^3
+81/8e_{12}(-4)e_{12}(-1)^2
+135/8e_1(-3)e_2(-1)e_{12}(-1)^2\\&
-9/16h_1(-1)^3e_{12}(-1)^3
-9h_2(-3)e_{12}(-1)^3+9/16h_2(-1)^3e_{12}(-1)^3\\&
+9h_1(-3)e_{12}(-1)^3
-54e_1(-2)e_2(-2)e_{12}(-1)^2+18f_2(-1)e_2(-2)e_{12}(-1)^3\\&
+9/4f_1(-2)e_1(-1)e_{12}(-1)^3
+135/8e_1(-1)e_2(-3) e_{12}(-1)^2\\&
-2889/64e_1(-1) e_2(-1) e_{12}(-2)^2
-1053/32e_{12}(-3)e_{12}(-2)e_{12}(-1)\\&
-18f_1(-1)e_1(-2)e_{12}(-1)^3
+9/4f_{12}(-2)e_{12}(-1)^2e_{12}(-1)^2\\&
+9/8f_{12}(-1)e_{12}(-2)e_{12}(-1)^3
-9/4f_2(-2)e_2(-1)e_{12}(-1)^3\\&
-135/32h_1(-2)e_{12}(-2)e_{12}(-1)^2
-135/16h_1(-2)h_1(-1)e_{12}(-1)^3\\&
+171/16h_1(-2)h_2(-1)e_{12}(-1)^3
+459/16h_1(-1)e_{12}(-3)e_{12}(-1)^2\\&
-2133/64h_1(-1)e_{12}(-2)^2e_{12}(-1)
-171/16h_1(-1)h_2(-2)e_{12}(-1)^3\\&
+9/16h_1(-1)h_2(-1)^2e_{12}(-1)^3
+297/32h_1(-1)^2e_{12}(-2)e_{12}(-1)^2\\&
-9/16h_1(-1)^2h_2(-1)e_{12}(-1)^3
+405/32h_2(-2)e_{12}(-2)e_{12}(-1)^2\\&
+135/16h_2(-2)h_2(-1)e_{12}(-1)^3
-243/16h_2(-1)e_{12}(-3)e_{12}(-1)^2\\&
+729/64h_2(-1)e_{12}(-2)^2e_{12}(-1)
-243/32h_2(-1)^2e_{12}(-2)e_{12}(-1)^2\\&
+135/4e_1(-2)e_2(-1)e_{12}(-2)e_{12}(-1)
+135/4e_1(-1)e_2(-2)e_{12}(-2)e_{12}(-1)\\&
-27/16e_1(-1)e_2(-1)e_{12}(-3)e_{12}(-1)
+135/8f_1(-1)e_1(-1)e_{12}(-2)e_{12}(-1)^2\\&
+9/2f_1(-1)f_2(-1)e_{12}(-1)^2e_{12}(-1)^2
-9/4f_1(-1)h_1(-1)e_1(-1)e_{12}(-1)^3\\&
-9/4f_1(-1)h_2(-1)e_1(-1)e_{12}(-1)^3
-9/4f_{12}(-1)e_1(-1)e_2(-1)e_{12}(-1)^3\\&
-9/4f_{12}(-1)h_1(-1)e_{12}(-1)^2e_{12}(-1)^2
+9/4f_{12}(-1)h_2(-1)e_{12}(-1)^2e_{12}(-1)^2\\&
-135/8f_2(-1)e_2(-1)e_{12}(-2)e_{12}(-1)^2
+9/4f_2(-1)h_1(-1)e_2(-1)e_{12}(-1)^3\\&
+9/4f_2(-1)h_2(-1)e_2(-1)e_{12}(-1)^3
-135/16h_1(-2)e_1(-1)e_2(-1)e_{12}(-1)^2\\&
-27/4h_1(-1)e_1(-2)e_2(-1)e_{12}(-1)^2
-27/4h_1(-1)e_1(-1)e_2(-2)e_{12}(-1)^2\\&
+27/16h_1(-1)h_2(-1)e_{12}(-2)e_{12}(-1)^2
-27/16h_1(-1)^2e_1(-1)e_2(-1)e_{12}(-1)^2\\&
-135/16h_2(-2)e_1(-1)e_2(-1)e_{12}(-1)^2
-27/4h_2(-1)e_1(-2)e_2(-1)e_{12}(-1)^2\\&
-27/4h_2(-1)e_1(-1)e_2(-2)e_{12}(-1)^2
-27/16h_2(-1)^2e_1(-1)e_2(-1)e_{12}(-1)^2\\&
+351/16h_1(-1)e_1(-1)e_2(-1)e_{12}(-2)e_{12}(-1)\\&
-27/8h_1(-1)h_2(-1)e_1(-1)e_2(-1)e_{12}(-1)^2\\&
+351/16 h_2(-1)e_1(-1)e_2(-1)e_{12}(-2)e_{12}(-1))\mathbf{1}
\end{align*}
is a singular vector of $\widetilde{\mathfrak{g}}$-module $V(\frac{1}{2},\mathbb{C})$
with weight $r_{\alpha_0+\delta}.(\frac{1}{2}\Lambda_0)$.
Then we have
$$A(L_{\widehat{\mathfrak{g}}}(\frac{1}{2},0))\cong U(\mathfrak{g})\Big/\langle F(\overline{v_2})\rangle,$$ where
\begin{align*}
F(\overline{v_2})=&27/128e_{12}^3
+9/64e_{12}^3h_1
-9/64e_{12}^3h_2
-27/64e_{12}^2e_2e_1
+9/8e_{12}^3e_1f_1
-27/8e_{12}^4f_{12}\\&
-9/8e_{12}^3e_2f_2
-27/32e_{12}^3h_1^2
-27/16e_{12}^3h_2h_1
-27/32e_{12}^3h_2^2
+9/4e_{12}^3e_1h_1f_1\\&
+9/4e_{12}^3e_1h_2f_1
-9/2e_{12}^4f_2f_1
-9/4e_{12}^4h_1f_{12}
+9/4e_{12}^4h_2f_{12}
+9/4e_{12}^3e_2e_1f_{12}\\&
-9/4e_{12}^3e_2h_1f_2
-9/4e_{12}^3e_2h_2f_2
-9/16e_{12}^3h_1^3
-9/16e_{12}^3h_2h_1^2
+9/16e_{12}^3h_2^2h_1\\&
+9/16e_{12}^3h_2^3
+27/16e_{12}^2e_2e_1h_1^2
+27/8e_{12}^2e_2e_1h_2h_1
+27/16e_{12}^2e_2e_1h_2^2.
\end{align*}

\begin{lem}\label{lem11}
$P_0=\mbox{span}_{\mathbb{C}}\{p_1^\prime(h_1,h_2),p_2^\prime(h_1,h_2)\}$, where
\begin{align*}
p_1^\prime(h_1,h_2)=&-81/32h_1+135/64h_1^2+729/64h_1^3-297/32h_1^4-81/16h_1^5+27/8h_1^6+81/32h_2\\&
-135/64h_2^2-729/64h_2^3+297/32h_2^4+81/16h_2^5-27/8h_2^6-729/64h_1h_2^2\\&
+297/16h_1h_2^3+243/16h_1h_2^4-27/2h_1h_2^5+729/64h_1^2h_2+81/8h_1^2h_2^3\\&-135/8h_1^2h_2^4
-297/16h_1^3h_2-81/8h_1^3h_2^2-243/16h_1^4h_2+135/8h_1^4h_2^2+27/2h_1^5h_2,\\
p_2^\prime(h_1,h_2)=&81/32h_1-135/64h_1^2-729/64h_1^3+297/32h_1^4+81/16h_1^5-27/8h_1^6-459/64h_1h_2\\&
-243/64h_1h_2^2+27h_1h_2^3-405/16h_1h_2^4+27/4h_1h_2^5-243/16h_1^2h_2\\&+2025/32h_1^2h_2^2
-567/8h_1^2h_2^3+189/8h_1^2h_2^4+729/16h_1^3h_2-243/4h_1^3h_2^2+27h_1^3h_2^3\\&-81/8h_1^4h_2
+27/4h_1^4h_2^2-27/4h_1^5h_2.
\end{align*}
\end{lem}

Then the set $\{L(\lambda)|\lambda=0,-\frac{3}{2}\Lambda_1,-\frac{3}{2}\Lambda_2,\frac{3}{2}\Lambda_1+\frac{3}{2}\Lambda_2,\alpha\Lambda_1+(-\frac{1}{2}-\alpha)\Lambda_2,
\alpha\Lambda_1+(\frac{1}{2}-\alpha)\Lambda_2,\alpha\Lambda_1+(1-\alpha)\Lambda_2,\alpha\Lambda_1+(2-\alpha)\Lambda_2,\alpha\in\mathbb{C}\}$
provides the complete list of irreducible $A(L_{\widehat{\mathfrak{g}}}(\frac{1}{2},0))$-modules from the category $\mathcal{O}_{\mathfrak{g}}$.
Hence we obtain all irreducible weak $L_{\widehat{\mathfrak{g}}}(\frac{1}{2},0)$-modules in the category $\mathcal{O}_{\frac{1}{2}}$ and all irreducible ordinary $L_{\widehat{\mathfrak{g}}}(\frac{1}{2},0)$-modules.

\begin{thm}
The set $\{L(\frac{1}{2},\lambda)|\lambda=0,-\frac{3}{2}\Lambda_1,-\frac{3}{2}\Lambda_2,\frac{3}{2}\Lambda_1+\frac{3}{2}\Lambda_2,\alpha\Lambda_1+(-\frac{1}{2}-\alpha)\Lambda_2,
\alpha\Lambda_1+(\frac{1}{2}-\alpha)\Lambda_2,\alpha\Lambda_1+(1-\alpha)\Lambda_2,\alpha\Lambda_1+(2-\alpha)\Lambda_2,\alpha\in\mathbb{C}\}$
provides the complete list of irreducible modules in the category $\mathcal{O}_{\frac{1}{2}}$.
Moreover, the set $\{L(\frac{1}{2},\lambda)|\lambda=0,\frac{3}{2}\Lambda_1+\frac{3}{2}\Lambda_2,
\alpha\Lambda_1+(1-\alpha)\Lambda_2,\alpha\Lambda_1+(2-\alpha)\Lambda_2,\alpha\in\mathbb{C}\}$
provides the complete list of irreducible ordinary $L_{\widehat{\mathfrak{g}}}(\frac{1}{2},0)$-modules.
\end{thm}

Similar to the Proposition \ref{propmodule}, we have that the $\widetilde{\mathfrak{g}}$-module $L(\frac{1}{2},U)$ is a weak $L_{\widehat{\mathfrak{g}}}(\frac{1}{2},0)$-module if and only if $p_1^\prime(h_1,h_2)U$ $=p_2^\prime(h_1,h_2)U=0$.
Therefore, similar to the Theorem \ref{thmav}, we have that
$A_{\omega_\xi}(L_{\widehat{\mathfrak{g}}}(\frac{1}{2},0))$
is isomorphic to the quotient algebra $\mathbb{C}[t_1,t_2]/\langle p_1^\prime(t_1,t_2),p_2^\prime(t_1,t_2)\rangle$.
From the Finiteness Theorem (cf.~\cite{CLO}), we have that $A_{\omega_\xi}(L_{\widehat{\mathfrak{g}}}(\frac{1}{2},0))$ is infinite-dimensional.
Hence from the Proposition \ref{prop22}, $(L_{\widehat{\mathfrak{g}}}(\frac{1}{2},0),\omega_\xi)$ is not rational.
Suppose $\mathcal{C}_{\frac{1}{2}}$ is semisimple,
similar to the Theorem \ref{thm22},
$(L_{\widehat{\mathfrak{g}}}(\frac{1}{2},0),\omega_\xi)$ is rational,
it is a contradiction,
hence $\mathcal{C}_{\frac{1}{2}}$ is not semisimple.
Thus we have the following theorem.

\begin{thm}
The $\mathbb{Q}$-graded vertex operator superalgebra $(L_{\widehat{\mathfrak{g}}}(\frac{1}{2},0),\omega_\xi)$ is not rational.
Moreover, the category $\mathcal{C}_{\frac{1}{2}}$ is not semisimple.
\end{thm}

In the end, we present the following conjecture for the vertex operator superalgebra $L_{\widehat{\mathfrak{g}}}(\mathcal{k},0)$
at boundary admissible level $\mathcal{k}$.

\begin{con}
Let $\mathcal{k}=-\frac{m}{m+1}~(m\in\mathbb{N})$ be the boundary admissible level.
Then $L_{\widehat{\mathfrak{g}}}(\mathcal{k},0)$ is rational in the category $\mathcal{O}$
and the irreducible weak modules in the category $\mathcal{O}$ are exactly the admissible modules of level $\mathcal{k}$ for $\widehat{\mathfrak{g}}$.
Moreover,
the $\mathbb{Q}$-graded vertex operator superalgebra $(L_{\widehat{\mathfrak{g}}}(\mathcal{k},0),\omega_\xi)$ is rational and $C_2$-cofinite.
\end{con}

\begin{rem}
{\em Let $\mathfrak{g}$ be any simple finite-dimensional Lie algebra or $osp(1|2n)$, we know that
$L_{\widehat{\mathfrak{g}}}(\mathcal{k},0)$ is rational in the category $\mathcal{O}$ at admissible level $\mathcal{k}$
and the irreducible weak modules in the category $\mathcal{O}$ are exactly the admissible modules of level $\mathcal{k}$ for $\widehat{\mathfrak{g}}$ \cite{A,GS}.
For the superalgebra case, we also believe that this conjecture holds for any basic simple Lie superalgebras except $osp(1|2n)$.
We also want to mention that we believe that the conclusion of this conjecture holds only if $\mathcal{k}$ is the boundary admissible level.}
\end{rem}

\section*{Acknowledgements}
We would like to thank Professor Maria Gorelik for sending us another proof of the Proposition \ref{lemsv1},
which is important in the computation of the Zhu algebra $A(L_{\widehat{\mathfrak{g}}}(\mathcal{k},0))$.
Q. Wang is partially supported by China NSF grants (No. 12571033).

\end{document}